\documentclass[11pt,reqno, a4paper]{amsart}
\usepackage{amssymb,amsfonts,amsmath,bm}
\usepackage{hyperref}
\usepackage{psfrag,graphicx,color}
\usepackage{mathrsfs}
\usepackage{graphics}
\usepackage{subfigure}
\usepackage{enumerate,enumitem}
\usepackage{pgfplots}
\usetikzlibrary{arrows.meta}
\usepackage{MnSymbol}
\usepackage{cite}
\usepackage{diagbox}

\allowdisplaybreaks

\usepackage{booktabs}
\usepackage{multirow}

\usepackage[margin=1.05in]{geometry}

\DeclareGraphicsExtensions{.eps}
\usepackage{amssymb,amsmath,graphicx,amsfonts,euscript,mathrsfs}
\usepackage{hyperref}

\usepackage{titlesec} 
\titleformat{\section}{\vskip10pt\large\bfseries}{\thesection.}{0.5em}{\centering\vspace{5pt}}
\titleformat{\subsection}{\vskip10pt\normalsize\bfseries}{\thesubsection.}{0.5em}{}

\newtheorem{theorem}{Theorem}[section]
\newtheorem{lemma}[theorem]{Lemma}

\theoremstyle{definition}

\newtheorem{example}[theorem]{Example}

\def\R{{\mathbb R}}

\def\d{{\mathrm d}}

\def\endproof{\qed}

\numberwithin{equation}{section}

\begin{document}

\title[]{Analysis of fully discrete finite element methods for 2D Navier--Stokes equations with critical initial data }

\author[]{\,\,Buyang Li}
\address{\hspace*{-12pt}Buyang Li: 
Department of Applied Mathematics, The Hong Kong Polytechnic University,
Hung Hom, Hong Kong. {\it E-mail address}: {\tt buyang.li@polyu.edu.hk} }

\author[]{\,\,Shu Ma}
\address{\hspace*{-12pt}Shu Ma: 
Department of Applied Mathematics, The Hong Kong Polytechnic University,
Hung Hom, Hong Kong. {\it E-mail address}: {\tt maisie.ma@connect.polyu.hk} }

\author[]{\,\,Yuki Ueda}
\address{\hspace*{-12pt}Yuki Ueda: Waseda Research Institute for Science and Engineering, Faculty of Science and Engineering, Waseda University, Japan. {\it E-mail address}: {\tt yuki.ueda@aoni.waseda.jp}}




\maketitle
\begin{abstract}\noindent 
First-order convergence in time and space is proved for a fully discrete semi-implicit finite element method for the two-dimensional Navier--Stokes equations with $L^2$ initial data in convex polygonal domains, without extra regularity assumptions or grid-ratio conditions. The proof utilises the smoothing properties of the Navier--Stokes equations in the analysis of the consistency errors, an appropriate duality argument, and the smallness of the numerical solution in the discrete $L^2(0,t_m;H^1)$ norm when $t_m$ is smaller than some constant. Numerical examples are provided to support the theoretical analysis. 
\end{abstract}



\section{\bf Introduction}\label{sec:intr}


We consider the initial and boundary value problem of the incompressible Navier--Stokes (NS) equations  
\begin{equation}
\label{pde}
\left \{
\begin{aligned} 
\partial_t u + u\cdot\nabla u -\varDelta u
+\nabla p &= 0 && \mbox{in}\,\,\,  \varOmega\times (0,T] ,\\
\nabla\cdot u&=0&&\mbox{in}\,\,\,  \varOmega\times (0,T] ,\\
u&=0 && \mbox{on}\,\,\, \partial\varOmega\times (0,T] ,\\
u&=u^0 && \mbox{at}\,\,\, \varOmega\times \{0\} , 
\end{aligned}
\right .
\end{equation}
in a convex polygon $\varOmega\subset\R^2$ with boundary $\partial\varOmega$, up to a given time $T>0$. 
It is known that for any given initial value 
$$
u^0\in \dot L^2=\{v\in L^2(\Omega)^2: \nabla\cdot v=0\,\,\,\mbox{and}\,\,\, v\cdot \nu=0\,\,\,\mbox{on}\,\,\,\partial\Omega\} 
$$ 
(where $\nu$ denotes the unit normal vector on $\partial\Omega$), 
problem \eqref{pde} has a unique weak solution $u\in L^2(0,T;\dot H_0^1)\cap H^1(0,T;\dot H^{-1})\hookrightarrow C([0,T];\dot L^2)$, where 
$$
\dot H_0^1 =\{v\in H^1_0(\Omega)^2:\nabla\cdot v=0\} \quad\mbox{and}\quad
\mbox{$\dot H^{-1}$ denotes the dual space of $\dot H_0^1$};
$$
see \cite[Theorem 3.2 of Chapter 3]{Temam-1977}.

Stability and convergence of the numerical solution to the NS equations \eqref{pde} were studied based on different regularity assumptions on the solution and initial data. In particular, if the initial data are sufficiently smooth, i.e. $u^0\in \dot H_0^1\cap H^2(\Omega)^2$ or above, then the numerical solution was proved to be convergent with optimal order for finite element and spectral Galerkin methods with different time-stepping schemes, including the Crank--Nicolson method \cite{Heywood-Rannacher-1990}, the implicit-explicit Crank--Nicolson/Adams--Bashforth method \cite{marion-1998-navier,tone-2004-error,he-2007-stability}, the semi-implicit Crank-Nicolson extrapolation method \cite{baker-1976-galerkin, guo-2018-unconditional, sun-2011-error, ingram-2013-new}, 
the 
stabilization methods based on the Crank--Nicolson method \cite{ervin-2012-numerical}, the Crank--Nicolson extrapolation method \cite{labovsky-2009-stabilized}, the backward Euler method \cite{decaria-2018-time} and the BDF2 method \cite{layton-2014-numerical}, 
the projection methods \cite{shen-1992-error-1, shen-1992-error-2, achdou-2000-convergence},
the projection-based variational multiscale methods based on the Crank--Nicolson method \cite{shan-2013-numerical},
the fractional-step methods \cite{badia-2007-convergence}, 
the three-step implicit-explicit backward extrapolating scheme \cite{baker-1982-higher,wang-2012-convergence}, the backward differentiation formulae \cite{badia-2006-analysis, bermejo-2012-second, de-2009-postprocessing}, 
a second order energy- and helicity-preserving method \cite{rebholz-2007-energy}, 
the implicit-explicit Euler method \cite{marion-1998-navier},
and the implicit Euler methods with different spatial discretizations \cite{de-2009-postprocessing, de-2019-error,garcia-2020-symmetric, kaya-2005-discontinuous}.  The error estimates in the above-mentioned articles do not apply to nonsmooth initial data.

If the initial value $u^0$ is only in $\dot H_0^1$, the accuracy of second- or higher-order time-stepping schemes for the NS equations is often reduced. 
For the semidiscrete finite element method (FEM), it was shown in \cite{hill-2000-approximation} that the $L^2$-norm error bound at time $t$ is of $O(t^{-1/2}h^2)$, where $h$ denotes the mesh size of the finite elements. For fully discrete FEMs, 
the linearized Crank--Nicolson scheme was proved $1.5$th-order convergent in time \cite{he-2012-crank}, 
and the semi-implicit Euler scheme with spectral Galerkin method was proved first-order convergent under a CFL condition $ \tau \ln(\lambda_m/\lambda_1) \le \kappa$ in \cite{he-2005-stability},
where $\tau$ is the time stepsize and $\lambda_m$ is the maximal eigenvalue of the Stokes operator used by the spectral method (and $\kappa$ is a positive constant). 

%
%

The error estimates in the above-mentioned articles all require the initial value to be strictly smoother than $\dot L^2$, which is known to be a critical space for the 2D NS equations, a maximal Sobolev space on which well-posedness of the 2D NS equations is proved; see \cite{Gallagher2016}. 
As a result, the error analysis in this case turns out to be much more challenging than that for smoother initial data. To the best of our knowledge, for $\dot L^2$ initial data, the only error estimate for the NS equations is in \cite{he-2008-stability} for a semi-implicit Euler scheme with a spectral Galerkin method in space using the eigenfunctions and eigenvalues of the continuous Stokes operator. 
Under a CFL condition $\tau \le \kappa \lambda_M^{-1}$, it is shown in \cite{he-2008-stability} that the backward Euler spectral Galerkin method with time stepsize $\tau$ and maximal eigenvalue $\lambda_M$ has an error bound $O(\lambda_M^{-1/2}+\tau^{1/2})$ on a bounded time interval. 
For the backward Euler scheme with finite element methods (FEMs) in space, several stability results were proved in \cite{he-2007-stabilized} without error estimates. 
Overall, first-order convergence of the implicit or semi-implicit Euler methods and error estimates of fully discrete FEMs for the NS equations with $\dot L^2$ initial data still remains open. 


In this article, we prove the first-order convergence of a fully discrete FEM for the 2D NS equations with $\dot L^2$ initial data, using semi-implicit Euler method in time and an inf-sup stable pair of finite element spaces with divergence-free velocity field, i.e.,
\bigskip
\begin{align*}
\|u_h^n-u(t_n) \|_{L^2} 
\le 
C  (t_n^{-1} \tau_n + t_n^{-\frac12} h )
\quad \mbox{for}\,\,\,
t_n\in(0,T] , \\[-5pt]
\end{align*} 
where $u_h^n$ denotes the numerical solution at time level $t=t_n$. 
The main difficulty in analysing numerical methods for the NS equations with $\dot L^2$ initial data is to control the nonlinear terms appearing in the error analysis by very weak bounds of the numerical solution, in the presence of singular consistency errors (see Lemma \ref{LM:Consistency-Euler}). 
We overcome these difficulties by utilising the $O(t^{m})$-weighted $L^2$ estimates of the $m$th-order time derivative and $2m$th-order spatial derivatives (as shown in Lemma \ref{lemma:regularity-u}) and a duality argument with variable temporal stepsizes to resolve the initial singularity in the consistency errors (as shown in Section \ref{section:duality}). 
It is known that variable stepsizes can help resolve the singularity in proving convergence of exponential integrators for semilinear parabolic equations with nonsmooth initial data; see \cite{Li-Ma-JSC}. However, the error analysis for the NS equations turns out to be completely different from the error analysis for the semilinear parabolic equation due to the lack of Lipschitz continuity of the nonlinearity and the critical nature of the $\dot L^2$ space. This leads to the critical difficulty in the use of duality argument --- the lack of stronger norms than $L^2(0,T;H^1)$ to help control the nonlinear terms, as shown in \eqref{L2-L2-ehn}, where the the first term on the right-hand side of \eqref{L2-L2-ehn} has to be absorbed by the left-hand side. This difficulty is overcome by proving the smallness of the numerical solution in the discrete $L^2(0,t_m;H^1)$ norm when $t_m$ is smaller than some constant independent of the stepsize $\tau$ and mesh size $h$, as shown in Lemma \ref{Lemma:L2L2uhn}. 

The rest of this article is organised as follows. In Section \ref{sec:main_results}, we describe the finite element method and time-stepping scheme to be analysed in this article, and present the main theorem of this article. 
%
%
The proof of the main theorem is presented in Section \ref{sec:convergence_fully_FEM_Euler}. 
%
The proof of two technical lemmas, including the temporally weighted regularity results and the strong and weak convergence of the numerical solution, are presented in Appendix. 


\newpage
\section{The main result}\label{sec:main_results}

For $s\in\R$ and $1\le p\le \infty$, we denote by $W^{s,p}(\varOmega)$ the conventional Sobolev spaces of functions defined on $\varOmega$, with abbreviation $H^s(\varOmega)=W^{s,2}(\varOmega)$ and $L^p(\varOmega)=W^{0,p}(\varOmega)$. 
For the simplicity of notation, we denote by $\|\cdot\|_{W^{s,p}}$ the norm of the spaces $W^{s,p}(\varOmega)$, $W^{s,p}(\varOmega)^2$ and $W^{s,p}(\varOmega)^{2\times 2}$, omitting the dependence on $\varOmega$ and dimension. 

The natural function spaces associated to incompressible flow are the divergence-free subspaces of $L^2(\varOmega)^2$ and $H^1_0(\varOmega)^2$, denoted by 
$$X=\dot L^2\quad\mbox{and}\quad V=\dot H^1_0 ,$$ 
respectively, as defined in the introduction section. 
We denote by $P_X$ the $L^2$-orthogonal projection from $L^2(\varOmega)^2$ onto $\dot L^2$, and denote by 
$$A=P_X\Delta$$ the Stokes operator on $\dot L^2$ with domain $D(A) = \dot H^1_0\cap H^2(\varOmega)^2$, which is a self-adjoint operator on $\dot L^2$ and bounded above by $0$ (negative definite). 
The Stokes operator has an extension as a bounded operator $A: \dot H^1_0\rightarrow \dot H^{-1}$ defined by 
\begin{align}\label{def-Stokes-A}
(Av,w) =-\int_{\varOmega} \nabla v\cdot \nabla w\,\d x \quad\forall\, v,w\in \dot H^1_0 .
\end{align}
Correspondingly, the NS equations \eqref{pde} can be equivalently written into the abstract form: 
\begin{equation}
\label{pde_abstract}
\left \{
\begin{aligned} 
\partial_tu(t) + P_X(u(t)\cdot\nabla u(t)) - Au(t) &= 0  \quad \mbox{for}\,\,\, t\in(0,T],\\
u(0)&=u^0 . 
\end{aligned}
\right .
\end{equation}

Henceforth we use the common notation $(\cdot,\cdot)$ to denote the inner products of the Hilbert spaces $L^2(\varOmega)$, $L^2(\varOmega)^2$ and $L^2(\varOmega)^{2\times 2}$, and define 
$$L^2_0(\varOmega):=\{ v\in L^2(\varOmega): \mbox{$\int_{\varOmega}$} v\,\d x=0\} .$$ 

Let $V_h\times Q_h \subset H^1_0(\varOmega)^2\times L^2_0(\varOmega)$ be a pair of finite element spaces with the following properties: 
\begin{enumerate}[label={\rm(\arabic*)},ref=\arabic*,topsep=2pt,itemsep=0pt,partopsep=1pt,parsep=1ex,leftmargin=25pt]

\item[(P1)] There exists a linear projection operator $\Pi_h: H^1_0(\varOmega)^2\rightarrow V_h$ such that 
\begin{enumerate}
\item[(i)]
$\nabla\cdot \Pi_hv = P_{Q_h} \nabla\cdot v$ for $v\in H^1_0(\varOmega)^2$, where $P_{Q_h}:L^2_0(\varOmega)\rightarrow Q_h $ denotes the $L^2$-orthogonal projection. 

\item[(ii)]
The following approximation property holds for $v\in H^1_0(\varOmega)^2\cap H^m(\varOmega)^2$:
\begin{align}\label{Fortin-Error} 
\|v-\Pi_hv\|_{H^{s}(\varOmega)} 
\le Ch^{m-s}\|v\|_{H^{m}(\varOmega)},\quad 0\le s\le 1,\quad 1\le m\le 2.
\end{align}

\end{enumerate}

\item[(P2)] $\nabla\cdot v_h\in Q_h $ for $v_h\in V_h$. 

\end{enumerate}
The two properties above guarantee the inf-sup condition for the pair $V_h\times Q_h $  (see \cite{2014-Guzman-Neilan}), i.e., 
\begin{align}\label{inf-sup}
\|q_h\|_{L^2(\varOmega)}
\le 
\sup_{v_h\in V_h \backslash \{0\} }
\frac{C(\nabla\cdot v_h,q_h)}{\|v_h\|_{H^1}} \quad\forall\,q_h\in Q_h .
\end{align}

Since $\nabla\cdot v_h\in Q_h $ for $v_h\in V_h$, it follows that the discrete divergence-free subspace of $V_h$ coincides with its pointwise divergence-free subspace, i.e., 
\begin{align}\label{def-Xh}
X_h:=\{v_h\in V_h: (\nabla\cdot v_h, q_h) =0\,\,\, \forall\, q_h\in Q_h \} 
= \{v_h\in V_h: \nabla\cdot v_h=0\}  .
\end{align}
Hence, $$V_h\subset H^1_0(\Omega)^2\quad\mbox{and}\quad X_h\subset \dot H^1_0\subset \dot L^2 = X ,\quad\mbox{but}\quad 
V_h\not\subset \dot H^1_0 .$$ 

There exists several finite element spaces $V_h\times Q_h $ satisfying properties (P1)--(P2). An example was constructed in \cite{2014-Guzman-Neilan}, where $V_h$ consists of piecewise linear polynomials plus quadratic bubble functions, and $Q_h$ is simply the space of piecewise constants with vanishing integral over $\varOmega$. 
Another example is the Scott--Vogelius element space proposed in \cite{scott-1985-norm}, where $V_h$ is the space of continuous piecewise polynomials of degree $k\ge 4$, and $Q_h$ is the space of discontinuous piecewise polynomials of degree $k-1$. The Scott--Vogelius element space was proved to be inf-sup stable in \cite{Guzman-Scott-2019}, and therefore property (P1) is satisfied if $\Pi_h$ is simply chosen to be the Stokes--Ritz projection; see \cite[Proposition 4.18]{Ern-Guermond-2004}. 

We consider the following semidiscrete FEM for \eqref{pde}: for given $u^0\in \dot L^2$, find 
$
(u_h,p_h)\in C^1([0,T];V_h)\times C([0,T];Q_h )
$ 
such that 
\begin{align}\label{FEM}
\left\{\begin{aligned}
(\partial_t u_h,v_h) + (u_h\cdot\nabla u_h,v_h) 
+(\nabla u_h,\nabla v_h) 
- (p_h,\nabla\cdot v_h) &= 0 \\[5pt] 
(\nabla\cdot u_h,q_h) &=0 \\[5pt]
u_h(0) &= u_h^0 : =P_{X_h}u^0 ,
\end{aligned}
\right.
\end{align}
for all test functions $(v_h,q_h)\in V_h\times Q_h $ and $t\in(0,T]$, where $P_{X_h}:\dot L^2\rightarrow X_h$ denotes the $L^2$-orthogonal projection onto $X_h$. This is equivalent to computing $u_h^0\in V_h$ by the weak formulation (with an auxiliary function $\eta_h\in Q_h$)
\begin{align}\label{FEM-u0h}
\left\{\begin{aligned}
(u^0-u_h^0,v_h) - (\eta_h,\nabla\cdot v_h) &= 0 \\
(\nabla\cdot u_h^0,q_h) &=0 
\quad\mbox{for all test functions}\,\,\, (v_h,q_h)\in V_h\times Q_h .
\end{aligned}
\right.
\end{align}
%
If we denote by $A_h:X_h\rightarrow X_h$ the discrete Stokes operator defined by 
\begin{align}\label{def-Ah}
(A_hv_h,w_h)=-(\nabla v_h,\nabla w_h),\quad\forall\,v_h,w_h\in X_h. 
\end{align}
Then the semi-discrete problem \eqref{FEM} is equivalent to find $u_h\in C^1([0,T];X_h)$ such that 
\begin{align}\label{AFEM}
\left\{\begin{aligned}
\partial_tu_h+P_{X_h}(u_h\cdot\nabla u_h)-A_hu_h&=0\quad\mbox{for}\,\,\, t\in(0,T], \\
u_h(0)&=u_h^0 . 
\end{aligned}\right.
\end{align}

Let $0=t_0<t_1<\cdots<t_N=T$ be a partition of the time interval $[0,T]$ with stepsize 
\begin{align}\label{stepsize}
\tau_1\sim\tau_2\quad\mbox{and}\quad \tau_n=t_n-t_{n-1} \sim (t_{n-1}/T)^\alpha \tau \,\,\,\,\mbox{for}\,\,\,\, 2\le n\le N,
\end{align}
where $\tau$ is the maximal stepsize, and `$\sim$' means equivalent magnitude (up to a constant multiple). The stepsizes defined in this way has the following properties:
\begin{enumerate}
\item
$\tau_n\sim \tau_{n-1}$ for two consecutive stepsizes. 

\item 
$\tau_1=\tau^{\frac{1}{1-\alpha}}$. Hence, the starting stepsize is much smaller than the maximal stepsize. This can resolve the solution's singularity at $t=0$. 

\item
The total number of time levels is $O(T/\tau)$. Hence, the total computational cost is equivalent to using a uniform stepsize $\tau$. 

\end{enumerate}



With the nonuniform stepsizes defined above,
we consider the following fully discrete semi-implicit Euler FEM: Find $(u_h^n,p_h^n)\in V_h\times Q_h$, $n=1,\dots,N$, such that 
\begin{align*}
\left\{\begin{aligned}
\bigg(\frac{ u_h^n - u_h^{n-1} }{\tau_n}, v_h\bigg)
+ (\nabla u_h^n, \nabla  v_h)
+ (u_h^{n-1}\cdot\nabla u_h^n, v_h) - (p_h^n,\nabla\cdot v_h) &= 0 &&\forall\, v_h\in V_h,   \\        
(\nabla\cdot u_h^n,q_h)&=0  &&\forall\, q_h\in Q_h,\\
u_h^0&=P_{X_h}u^0 . 
\end{aligned}\right.
\end{align*}
This is equivalent to finding $u_h^n\in X_h$, $n=1,\dots,N$, such that 
\begin{align}\label{fully-FEM-Euler}
\left\{\begin{aligned}
\bigg(\frac{ u_h^n - u_h^{n-1} }{\tau_n}, v_h\bigg)
+ (\nabla u_h^n, \nabla  v_h)
+ (u_h^{n-1}\cdot\nabla u_h^n, v_h) &= 0 \quad\forall\, v_h\in X_h,\,\,\,  \\        
u_h^0&=P_{X_h}u^0 , 
\end{aligned}\right.
\end{align}
which can also be written into an abstract form by using the operator $A_h$ defined in \eqref{def-Ah}, i.e.,
\begin{align}\label{semi-Euler}
\frac{ u_h^n - u_h^{n-1} }{\tau_n}
- A_h u_h^n 
+P_{X_h}(u_h^{n-1}\cdot\nabla u_h^n) = 0 
\quad\mbox{for}\,\,\, 1\le n\le N .  
\end{align}
In view of \eqref{def-Xh}, there holds $\nabla\cdot u_h^n=0$ and therefore $(u_h^{n-1}\cdot\nabla u_h^n,u_h^n) =0$ similarly as the continuous solution. As a result of this identity, substituting $v_h=u_h^n$ into \eqref{fully-FEM-Euler} immediately yields unconditional energy stability of the semidiscrete FEM.

The main result of this article is the following theorem, which provides the convergence of the fully discrete method \eqref{fully-FEM-Euler}.
\begin{theorem}\label{THM:FEM-Euler}
Let $u^0\in \dot L^2$ and assume that the finite element space $X_h\times Q_h$ has properties {\rm(P1)--(P2).} Then, when the time stepsizes satisfy \eqref{stepsize} with a fixed constant $\alpha\in(\frac12,1)$, the fully discrete solution given by \eqref{fully-FEM-Euler} has the following error bound: 
\begin{align}\label{Fully_Error_bound}
\|u_h^n-u(t_n) \|_{L^2} 
\le 
C (t_n^{-1}  \tau_n + t_n^{-\frac12}  h ),
\end{align}
where the constant $C$ depends only on $u^0$, $\Omega$ and $T$ {\rm(}independent of $t_n\in(0,T]$, $\tau$ and $h)$. 
\end{theorem}

The proof of Theorem \ref{THM:FEM-Euler} is presented in section~\ref{sec:convergence_fully_FEM_Euler}. For the simplicity of notation, we denote by $C$ a generic positive constant that may be different at different occurrences and may depend on $u^0$, $\Omega$ and $T$, but is independent of $\tau$ and $h$. 

\section{Proof of Theorem \ref{THM:FEM-Euler}}\label{sec:convergence_fully_FEM_Euler}

\subsection{Preliminary results}

In this subsection, we present some technical inequalities and regularity results that will be used in the error analysis for the NS equations. 

First, the following interpolation inequalities, which hold in general convex polygons (cf. \cite[Theorem 5.8, Theorem 5.9]{AdamsFournier2003}) and were often used in analysis of NS equations in the literature (e.g., \cite[Chapter III, \textsection 3.3]{Temam-1977}) and will also be used in this article: 
\begin{align}
&\|v\|_{L^4} \le C\|v\|_{L^2}^{\frac12} \|\nabla v\|_{L^2}^{\frac12} 
&&\hspace{-50pt}\forall\, v\in H^1_0(\varOmega) , 
\label{L4-L2-H1} \\
&\|\nabla v\|_{L^4} \le C\|\nabla v\|_{L^2}^{\frac12} \|\Delta v\|_{L^2}^{\frac12}  
&&\hspace{-50pt}\forall\, v\in H^1_0(\varOmega)\cap H^2(\varOmega),
\label{W14-H1-H2}\\
&\|v\|_{L^{\infty}} \leq C \|v\|_{L^2}^{\frac12}\|v\|_{H^2}^{\frac12}
&&\hspace{-50pt}
\forall\, v \in H_0^1(\Omega)\cap H^2(\Omega) .
\label{Linf-L2-H2}
\end{align}



Second, the following basic properties of finite element spaces and the finite element solution to the NS equations will be used. 
\begin{enumerate}
\item
Approximation of $X_h$ to $\dot H^1_0$: for $m=1,2$ there holds 
\begin{align}
\label{Xh-X-approx}
\inf_{v_h\in X_h} ( \|v-v_h\|_{L^2} + h \|v-v_h\|_{H^1} )
\le
Ch^m \|v\|_{H^m}
\quad\forall\, v\in \dot H^1_0\cap H^m(\Omega)^2. 
\end{align}
Since $\Pi_h v\in X_h$ for $v\in \dot H^1_0$, the inequality above follows from \eqref{Fortin-Error}.  

\item
$H^1$-stability of the $L^2$-orthogonal projection $P_{X_h}:\dot H^1_0\rightarrow X_h$:
\begin{equation}\label{H1-stability-PXh}
\|P_{X_h}v\|_{H^1}\le C\| v\|_{H^1} \mbox{ for all }v\in 
\dot H^1_0.
\end{equation}
{\it Proof.} 
By using the triangle inequality, inverse inequality and \eqref{Xh-X-approx}, we have 
\begin{align*}
\|v-P_{X_h}v\|_{H^1} 
&\le \inf_{v_h\in X_h}(\|v-v_h\|_{H^1}+\|P_{X_h}(v-v_h)\|_{H^1}) \\
&\le \inf_{v_h\in X_h}(\|v-v_h\|_{H^1}+Ch^{-1}\|v-v_h\|_{L^2}) 
\le C\|v\|_{H^1} .
\hspace{55pt}\mbox{\endproof}
\end{align*}

\item 
Error bound of the $L^2$-orthogonal projection $P_{X_h}:\dot H^1_0\rightarrow X_h$:
\begin{equation}\label{Error-PXh}
\|v - P_{X_h}v\|_{L^2} + h\|v - P_{X_h}v\|_{H^1}
\le
Ch^s \|v\|_{H^s} \quad\forall\, v\in \dot H^1_0\cap H^2(\Omega)^2 ,
\,\,\,s=1,2. 
\end{equation}
{\it Proof.} 
By using the triangle inequality, \eqref{Xh-X-approx} and \eqref{H1-stability-PXh}, we have 
\begin{align*}
&\|v - P_{X_h}v\|_{L^2} + h\|v - P_{X_h}v\|_{H^1}\\
&\le \inf_{v_h\in X_h}(\|v-v_h\|_{L^2} + h\|v-v_h\|_{H^1}+\|P_{X_h}(v-v_h)\|_{L^2} + h\|P_{X_h}(v-v_h)\|_{H^1}) \\
&\le \inf_{v_h\in X_h}(\|v-v_h\|_{L^2} + h\|v-v_h\|_{H^1}) 
\le Ch^s\|v\|_{H^s} .
\hspace{120pt}\mbox{\endproof}
\end{align*}

\item
The Stokes--Ritz projection and its error bound: 
Let $R_{X_h}:\dot H^1_0\rightarrow X_h$ be the Stokes--Ritz projection, defined by 
\begin{align}\label{def-Stokes-Ritz}
(\nabla(v-R_{X_h}v),\nabla w_h)=0\quad\forall\, w_h\in X_h,\,\,\,v\in \dot H^1_0 .
\end{align}
which is equivalent to finding $(R_{X_h}v, \eta_h)\in V_h\times Q_h$ such that 
\begin{align*}
\begin{aligned}
(\nabla(v-R_{X_h}v),\nabla w_h)-(\eta_h,\nabla\cdot w_h)&=0
&&\forall\, w_h\in V_h ,\\
(\nabla\cdot R_{X_h}v,q_h)&=0
&&\forall\, q_h\in Q_h .
\end{aligned}
\end{align*}
The Stokes--Ritz projection has the following error bound (cf. \cite[Lemma 2.44 and Lemma 2.45]{Ern-Guermond-2004} and \cite[Proposition 4.18]{Ern-Guermond-2004}):
\begin{align}\label{Error-Stokes-Ritz}
\|v-R_{X_h}v\|_{L^2} + h \|v-R_{X_h}v\|_{H^1} 
\le
Ch^s \|v\|_{H^s} \quad\forall\, v\in \dot H^1_0\cap H^2(\Omega)^2 ,\,\,\, 
s=1,2 .
\end{align}

%

\item
The numerical solution given by the fully discrete FEM in \eqref{fully-FEM-Euler} satisfies the following basic energy estimate: 
\begin{align}\label{FD-basic-energy}
\max_{1\le n\le N}\|u_h^{n}\|_{L^2}^2
+2\sum_{n=1}^N\tau_n \|\nabla u_h^{n}\|_{L^2}^2
\le
\|u_h^0\|_{L^2}^2 .
\end{align}
This can be obtained by testing \eqref{semi-Euler} with $u_h^n$.

\end{enumerate}
\medskip

Third, since inequality \eqref{W14-H1-H2} cannot hold for finite element functions (which do not have second-order partial derivatives), we would need the following discrete analogy of \eqref{W14-H1-H2}. 

\begin{lemma}\label{Lemma:discrete_W14}
The following inequality holds: 
\begin{equation}
\label{Eq:phih2}
\|\nabla\phi_h\|_{L^4}\le C\|\nabla \phi_h\|_{L^2}^{\frac12} \|A_h\phi_h\|_{L^2}^{\frac12} 
\quad\forall\,\phi_h\in X_h . 
\end{equation}
\end{lemma}
\begin{proof}
To obtain a bound for $\|\nabla\phi_h\|_{L^4}$, we let $\phi\in D(A) = \dot H^1_0\cap H^2(\varOmega)^2$ be the solution of
\begin{equation}
\label{Eq:phi}
A\phi = A_h\phi_h \quad\mbox{(where $A_h\phi_h \in X_h\subset X$)},
\end{equation}
which is equivalent to the linear Stokes equation
$$
\left\{\begin{aligned}
-\Delta \phi + \nabla \eta & = -A_h\phi_h &&\mbox{in}\,\,\,\Omega,\\
\nabla\cdot\phi &=0              &&\mbox{in}\,\,\,\Omega,\\
\phi &=0              &&\mbox{on}\,\,\,\partial\Omega . 
\end{aligned}\right.
$$
By the standard $H^2$ estimate of the linear Stokes equation (cf. \cite[Theorem 2]{Kellogg-Osborn-1976}), the solution $\phi\in D(A)$ satisfies 
\begin{align}\label{H2-phi-Ahphih}
\|\phi\|_{H^2}&\le C\|A_h\phi_h\|_{L^2} .
\end{align}
This inequality and \eqref{W14-H1-H2} imply that 
\begin{equation}\label{H2-phi-Ahphih-2}
\|\nabla\phi\|_{L^4} \le C\|\nabla\phi\|_{L^2}^{\frac12}\|A_h\phi_h\|_{L^2}^{\frac12} . 
\end{equation}

On the one hand, by using the projection operator $\Pi_h$ in \eqref{Fortin-Error}  and the standard interpolation inequality $\|f\|_{L^4}\le \|f\|_{L^2}^{\frac12}\|f\|_{L^\infty}^{\frac12}$, we have 
\begin{align}
\|\nabla(\phi_h-\Pi_h\phi)\|_{L^4}
\le& C\|\nabla(\phi_h-\Pi_h\phi)\|_{L^2}^{\frac12} 
\|\nabla(\phi_h-\Pi_h\phi)\|_{L^\infty}^{\frac12} \notag \\
\le& Ch^{-\frac12}\|\phi\|_{H^1}^{\frac12} 
\|\nabla(\phi_h-\Pi_h\phi)\|_{L^2}^{\frac12}
\quad\mbox{(inverse inequality)} 
\notag\\
\le& C\|\phi\|_{H^1}^{\frac12} \|\phi\|_{H^{2}}^{\frac12} \notag\\
\le& C\|\phi\|_{H^1}^{\frac12}\|A_h\phi_h\|_{L^2}^{\frac12},
\end{align}
where the last inequality uses \eqref{H2-phi-Ahphih}.
On the other hand, 
\begin{align}
\begin{aligned}
\|\nabla \Pi_h\phi\|_{L^4}
&\le C\|\nabla\phi\|_{L^4} 
&&\mbox{(stability of $\Pi_h$ in $W^{1,4}(\Omega)$)} \\
&\le C\|\phi\|_{H^1}^{\frac12}\|A_h\phi_h\|_{L^{2}}^{\frac12} .
&&\mbox{(here we have used \eqref{H2-phi-Ahphih-2})}
\end{aligned}
\end{align}
Combining the two estimates above and using the triangle inequality, we obtain 
\begin{equation}
\label{Eq:phih}
\|\nabla\phi_h\|_{L^4}\le C\|\phi\|_{H^1}^{\frac12}\|A_h\phi_h\|_{L^2}^{\frac12} . 
\end{equation}
It remains to prove the following inequality: 
\begin{align}
&\label{Eq:phih-22}
\|\phi\|_{H^1}\le C\|\phi_h\|_{H^1} .
\end{align}
Then substituting \eqref{Eq:phih-22} into \eqref{Eq:phih} yields the desired result \eqref{Eq:phih2}. 

In fact, testing equation \eqref{Eq:phi} by $\phi$ gives
\begin{align*}
\|\nabla\phi\|_{L^2}^2 = -(A_h\phi_h,\phi) = (\nabla\phi_h,\nabla P_{X_h}\phi)
\le C\|\nabla\phi_h\|_{L^2}\|P_{X_h}\phi\|_{H^1} 
\le C\|\nabla\phi_h\|_{L^2}\|\phi\|_{H^1} .
\end{align*}
where the last inequality uses the $H^1$ stability of the $L^2$ projection $P_{X_h}$, as shown in \eqref{H1-stability-PXh}. Since $\phi\in H^1_0(\Omega)^2$, it follows that $\|\phi\|_{H^1}\le C\|\nabla\phi\|_{L^2}$. Hence, the inequality above furthermore implies \eqref{Eq:phih-22}. 
This completes the proof of Lemma \ref{Lemma:discrete_W14}. 
\hfill\end{proof}

It is well known that the unique weak solution of the NS equations \eqref{pde} satisfies the energy equality
\begin{align}
\label{Eq:u_energy_eq-L2}
&\frac{1}{2}\|u(t)\|_{L^2}^2 + \|\nabla u\|_{L^2(0,t;L^2)}^2 = \frac{1}{2} \|u^0\|_{L^2}^2,\quad\forall\, t > 0 . 
\end{align}
This can be obtained by testing \eqref{pde} with $u$. In addition to this basic estimate, 
the following regularity result for the NS equations will be used in the error analysis.

\begin{lemma}\label{lemma:regularity-u}
For any given $u^0\in \dot L^2$, the solution of \eqref{pde} satisfies the following estimate:
\begin{align}
\label{Eq:ut_energy_eq}
&\|\partial_t^m u(t)\|_{L^2}  + t^{\frac12}\|\partial_t^m u(t)\|_{H^1} 
+ t\|\partial_t^m u(t)\|_{H^2} \le C_m t^{-m} \quad\forall\, t>0,\,\,\, m=0,1,2\dots
\end{align}
where the constant $C_m$ depends on $m$ and $\|u^0\|_{L^2}$. 
\end{lemma}
%


Since we have not found a proof of Lemma \ref{lemma:regularity-u} in the literature, we present a proof of this lemma in Appendix. 


\subsection{Estimates for the consistency errors}
We denote $\hat u_h^n= P_{X_h}u(t_n)$. By testing equation \eqref{pde_abstract} with $v_h\in X_h\subset X=\dot L^2$ and using the Stokes--Ritz projection operator defined in \eqref{def-Stokes-Ritz}, we have 
\begin{equation*}
(\partial_tP_{X_h}u(t),v_h)
+(u(t)\cdot\nabla u(t),v_h) + (\nabla R_{X_h} u(t),\nabla v_h) 
= 0 \quad\forall\, v_h\in X_h . 
\end{equation*}
which can be written into the abstract form  
\begin{equation*}
\partial_t P_{X_h}u(t)
+ P_{X_h}(u(t)\cdot\nabla u(t)) - A_h R_{X_h} u(t) 
= 0 . 
\end{equation*}
By considering this equation at $t=t_n$, we can write down the equation satisfied by $\hat u_h^n$, i.e., 
\begin{align}\label{FEM-int}
\begin{aligned}
\frac{\hat u_h^n - \hat u_h^{n-1} }{\tau_n}
+P_{X_h}(\hat u_h^{n-1}\cdot\nabla \hat u_h^n) 
- A_h \hat u_h^n = \mathcal{E}^n + \mathcal{F}_h^n
\quad\mbox{for}\,\,\, n\ge 1 , 
\end{aligned}
\end{align}
where the truncation errors $\mathcal{E}^n$ and $\mathcal{F}_h^n$ are given by
\begin{align}\label{eps-n-be}
\mathcal{E}^n 
=& 
\bigg(\frac{\hat u_h^n - \hat u_h^{n-1} }{\tau_n}-\partial_t \hat u_h(t_n)\bigg) 
+
P_{X_h}[(u(t_{n-1})-u(t_{n}))\cdot\nabla u(t_n)]  
=: \mathcal{E}_{1}^n + \mathcal{E}_{2}^n,\\
\label{F-h-n-be}
\mathcal{F}_h^n
=&
 - A_h(\hat u_h^n - R_{X_h} u(t_n))
+ P_{X_h}[ \hat u_h^{n-1} \cdot \nabla (\hat u_h^n-u(t_n))] \\
& 
+ P_{X_h}[(\hat u_h^{n-1} - u(t_{n-1})) \cdot \nabla u(t_n)] \notag\\
=&\!\!: \mathcal{F}_{h, 1}^n + \mathcal{F}_{h, 2}^n + \mathcal{F}_{h, 3}^n.
\notag
\end{align}
\begin{lemma}[Consistency errors]\label{LM:Consistency-Euler}
If $u^0\in \dot L^2$ and the stepsize in \eqref{stepsize} is used, then for all test functions $v_h\in X_h$, the consistency errors defined in \eqref{eps-n-be}--\eqref{F-h-n-be} satisfy the following estimates: 
\begin{align}\label{con_err-E}
&|(\mathcal{E}^n,v_h)| \le C\tau_n t_n^{-\frac 32}\|\nabla v_h\|_{L^2} 
\hspace{17.3pt}\mbox{for}\,\,\,\, n \ge 1, \\[5pt] 
\label{con_err-F}
&|(\mathcal{F}_h^n,v_h)|\le
\left\{
\begin{aligned}
&Cht_n^{-1}\|\nabla v_h\|_{L^2} &&\mbox{for}\,\,\,\,  n \ge 2, \\
&C(ht_n^{-1}+\tau_1 t_1^{-\frac 32})\|\nabla v_h\|_{L^2} &&\mbox{for}\,\,\,\,  n = 1. 
\end{aligned}
\right.
\end{align}
\end{lemma}
%
\begin{proof}
Testing \eqref{pde_abstract} by $v_h\in X_h$ and integrating the result over the time interval $(t_{n-1},t_n)$, we obtain 
\begin{align}\label{utn-utn-1}
(u(t_n) - u(t_{n-1}), v_h) = 
-\int_{t_{n-1}}^{t_n} (\nabla u(t), \nabla v_h) \, \d t 
- \int_{t_{n-1}}^{t_n} (u(t) \cdot \nabla u(t), v_h)\, \d t .
\end{align}
For $n=1$, the truncation errors can be estimated by using \eqref{utn-utn-1} and \eqref{pde_abstract}, and the triangle inequality: 
\begin{align*}
|(\mathcal{E}_{1}^1, v_h)| \le& 
\tau_1^{-1} |(u(t_1) - u(t_0), v_h)| + |(\partial_{t} u(t_1), v_h)|\\
=& \tau_1^{-1} \Bigl|\int_0^{t_1} (\nabla u(t), \nabla v_h)\, \d t - \int_0^{t_1} (u(t),u(t) \cdot \nabla v_h)\,\d t \Bigr| \\
&+ |(\nabla u(t_1), \nabla v_h) - (u(t_1) ,u(t_1)\cdot \nabla v_h)| \\
\le& C\tau_1^{-1} \int_0^{t_1} \bigl[\|u(t)\|_{H^1}\|\nabla v_h\|_{L^2}  + \|u(t)\|_{L^4}^2\|\nabla v_h\|_{L^2} \bigr] \, \d t\\
&+ C(\|u(t_1)\|_{H^1}+\|u(t_1)\|_{L^4}^2)\|\nabla v_h\|_{L^2}\\
\le& C\tau_1^{-1} \int_0^{t_1} \bigl[\|u(t)\|_{H^1}\|\nabla v_h\|_{L^2}  + \|u(t)\|_{L^2}\|u(t)\|_{H^1}\|\nabla v_h\|_{L^2} \bigr] \, \d t\\
&+ C(\|u(t_1)\|_{H^1}+\|u(t_1)\|_{L^2}\|u(t_1)\|_{H^1})\|\nabla v_h\|_{L^2}
\quad\mbox{ (here \eqref{L4-L2-H1} is used)}\\ 
\le& C\tau_1^{-\frac12} \|u\|_{L^2(0, t_1; H^1)}(1+\|u\|_{L^\infty(0, t_1; L^2)})\|\nabla v_h\|_{L^2}
+ C\tau_1^{-\frac12}\|\nabla v_h\|_{L^2}\\
&\hspace{214pt}\mbox{(here Lemma \ref{lemma:regularity-u} is used)} \\ 
\le&C\tau_1^{-\frac12}\|\nabla v_h\|_{L^2}
=C\tau_1t_1^{-\frac32} \|\nabla v_h\|_{L^2},
\end{align*}
and
\begin{align*}
|(\mathcal{E}_{2}^1, v_h)|
=&  |(u(t_1),(u(t_{0})-u(t_{1})) \cdot \nabla v_h)| \\
\le& C\|u(t_1)\|_{L^\infty} \|u(t_1)-u(t_0)\|_{L^2}\|\nabla v_h\|_{L^2} \\
\le&
C\|u(t_1)\|_{L^2}^{\frac12} \|u(t_1)\|_{H^2}^{\frac12} 
(\|u(t_0)\|_{L^2}+ \|u(t_1)\|_{L^2})
\|\nabla v_h\|_{L^2}  \quad\mbox{ (here \eqref{Linf-L2-H2} is used)}\\ 
\le& C\tau_1^{-\frac12} \|\nabla v_h\|_{L^2} \quad\mbox{(here Lemma \ref{lemma:regularity-u} is used)}\\ 
=& C\tau_1 t_1^{-\frac32} \|\nabla v_h\|_{L^2} . 
\end{align*}
By combining the two estimates above, we obtain 
$|(\mathcal{E}^n, v_h)| \le C\tau_n t_n^{-\frac32} \|\nabla v_h\|_{L^2}$ for $n=1$. 

In the case $n\ge 2$, by differentiating equation \eqref{pde_abstract} in time and testing the result by $v_h\in X_h$, we obtain
\begin{align*}
(\partial_{tt}u(t), v_h) + (\partial_{t} \nabla u(t), \nabla v_h) + (\partial_{t} u(t) \cdot \nabla u(t), v_h) + (u(t) \cdot \nabla \partial_{t} u(t), v_h)= 0 \quad \forall\, v_h \in X_h,
\end{align*}
which implies that
\begin{align*}
|(\mathcal{E}_{1}^n, v_h)|
=
&\bigg|\bigg(\frac{u(t_n) - u(t_{n-1}) }{\tau_n}-\partial_t u(t_n),v_h\bigg)\bigg| 
= \bigg|\bigg(\int_{t_{n-1}}^{t_n} \frac{t-t_{n-1}}{\tau_n} \partial_{tt}u(t)\d t , v_h\bigg) \bigg|\\
=& \bigg|\int_{t_{n-1}}^{t_n} \frac{t-t_{n-1}}{\tau_n} (\partial_{tt}u(t), v_h)\d t  \bigg|\\
\le&
C\tau_n 
\max_{t\in[t_{n-1},t_{n}]}|(\partial_{tt}u(t), v_h)|\\
=&
C\tau_n 
\max_{t\in[t_{n-1},t_{n}]} |(\partial_{t} \nabla u(t), \nabla v_h) - (u(t), \partial_{t} u(t) \cdot \nabla v_h) - ( \partial_{t} u(t), u(t) \cdot \nabla v_h)| \\
\le& C\tau_n 
\max_{t\in[t_{n-1},t_{n}]} \bigl[\|\partial_tu(t)\|_{H^1}
+ \|u_h(t)\|_{L^4}\|\partial_tu(t)\|_{L^4} \bigr]\|\nabla v_h\|_{L^2}  \\
\le& C\tau_n 
\max_{t\in[t_{n-1},t_{n}]} \bigl[ \|\partial_tu(t)\|_{H^1}
+  \|u(t)\|_{L^2}^\frac12 \|u(t)\|_{H^1}^\frac12 \|\partial_tu(t)\|_{L^2}^\frac12 \|\partial_tu(t)\|_{H^1}^\frac12  \bigr]\|\nabla v_h\|_{L^2}  \\
\le&C\tau_n t_{n-1}^{-\frac32} \|\nabla v_h\|_{L^2}
\quad\mbox{(here Lemma \ref{lemma:regularity-u} is used)}  , 
\end{align*}
and
\begin{align*}
|(\mathcal{E}_{2}^n, v_h)|
&=| ( u(t_n) , (u(t_{n-1})-u(t_{n}))\cdot\nabla v_h )| \\\
&\le C\|u(t_n)\|_{L^4} \|u(t_{n-1})-u(t_{n})\|_{L^4} \|\nabla v_h\|_{L^2} \\
&\le C\tau_n \max_{t\in[t_{n-1},t_{n}]} \bigl[\|u(t)\|_{L^4}  \|\partial_tu(t)\|_{L^4}\bigr]\|\nabla v_h\|_{L^2} \\
&\le C\tau_n \max_{t\in[t_{n-1},t_{n}]} \bigl[ \|u(t)\|_{L^2}^\frac12 \|u(t)\|_{H^1}^\frac12 \|\partial_tu(t)\|_{L^2}^\frac12 \|\partial_tu(t)\|_{H^1}^\frac12 \bigr] \|\nabla v_h\|_{L^2} \\
&\le
C\tau_n t_{n-1}^{-\frac32} \|\nabla v_h\|_{L^2} \quad\mbox{(here Lemma \ref{lemma:regularity-u} is used)}   .
\end{align*}
Since $t_{n-1}^{-\frac32} \sim t_n^{-\frac32}$ for $n \ge 2$, the two estimates above imply $|(\mathcal{E}^n, v_h)|\le C\tau_n t_n^{-\frac32} \|\nabla v_h\|_{L^2}$ for $n\ge 2$. 
This completes the proof of \eqref{con_err-E}. 

To prove \eqref{con_err-F}, we consider the expressions of $\mathcal{F}_{h,j}^n$, $j=1,2,3,$ defined in \eqref{F-h-n-be}. By using the approximation properties of the projection operators $P_{X_h}$ and $R_{X_h}$ in \eqref{Error-PXh} and \eqref{Error-Stokes-Ritz}, we have 
\begin{align*}
|(\mathcal{F}_{h, 1}^n, v_h)|
=&|( \nabla (P_{X_h}u(t_n) - R_{X_h} u(t_n)) , \nabla v_h )|\\
\le&\bigl( \|\nabla (P_{X_h} u(t_n)-  u(t_n))\|_{L^2} + \|\nabla (u(t_n) - R_{X_h} u(t_n))\|_{L^2} \bigr)  \|\nabla v_h\|_{L^2}\\
\le&Ch \|u(t_n)\|_{H^2}\|\nabla v_h\|_{L^2}\\
\le& Ch t_n^{-1}\|\nabla v_h\|_{L^2}
\quad  \mbox{for}\,\,\, v_h\in X_h \,\,\,\mbox{and} \,\,\, n \ge 1 ,
\end{align*}
where Lemma \ref{lemma:regularity-u} in deriving the last inequality. 

By applying the inverse inequality $\|v_h\|_{W^{1, 4}} \le Ch^{-\frac12} \|v_h\|_{W^{1, 2}}$ and \eqref{L4-L2-H1}, we obtain
\begin{align*}
|(\mathcal{F}_{h, 2}^n, v_h)| 
=& |(P_{X_h}u(t_{n-1})\cdot \nabla(P_{X_h}u(t_n)-u(t_n)),v_h)|  \\
=& |(P_{X_h}u(t_n)-u(t_n),P_{X_h}u(t_{n-1}) \cdot \nabla v_h )|  \\
\le&C\|P_{X_h}u(t_{n})-u(t_n)\|_{L^4} \|P_{X_h}u(t_{n-1})\|_{L^2} \|\nabla v_h\|_{L^4}\\
\le&C\|P_{X_h}u(t_{n})-u(t_n)\|_{L^2}^{\frac 12} \|P_{X_h}u(t_{n})- u(t_n)\|_{H^1}^{\frac 12}
\|u(t_{n-1})\|_{L^2} h^{-\frac12} \|\nabla v_h\|_{L^2}\\
\le&C\big( h^2\|u(t_n)\|_{H^2} \big)^{\frac12} 
\big( h\|u(t_n)\|_{H^2} \big)^{\frac12} \|u(t_{n-1})\|_{L^2} h^{-\frac12} \|\nabla v_h\|_{L^2}\\
\le& Ch\|u(t_n)\|_{H^2} \|u(t_{n-1})\|_{L^2} \|\nabla v_h\|_{L^2} \\
\le& Ch t_n^{-1} \|\nabla v_h\|_{L^2} 
\quad \mbox{for}\,\,\, v_h\in X_h \,\,\,\mbox{and} \,\,\, n \ge 1 ,
\end{align*} 
where Lemma \ref{lemma:regularity-u} is used in deriving the last inequality. 

Similarly as the estimate for $|(\mathcal{F}_{h, 2}^n, v_h)| $, we have 
\begin{align*}
|(\mathcal{F}_{h, 3}^n, v_h)|
=&|((P_{X_h} u(t_{n-1}) - u(t_{n-1})) \cdot \nabla u(t_n),v_h)| \\
=&|(u(t_n),(P_{X_h} u(t_{n-1}) - u(t_{n-1})) \cdot \nabla v_h)| \\
\le&C\|u(t_n)\|_{L^2} \|P_{X_h} u(t_{n-1}) - u(t_{n-1})\|_{L^4}\|\nabla v_h\|_{L^4}\\
\le& C\|u(t_n)\|_{L^2} \|P_{X_h} u(t_{n-1}) - u(t_{n-1})\|_{L^2}^{\frac12}
\|P_{X_h} u(t_{n-1}) - u(t_{n-1})\|_{H^1}^{\frac12} h^{-\frac 12}
\|\nabla v_h\|_{L^2}\\
\le& C\|u(t_n)\|_{L^2} (h^2\|u(t_{n-1})\|_{H^2})^{\frac12}
(h\|u(t_{n-1})\|_{H^2})^{\frac12} 
h^{-\frac 12} \|\nabla v_h\|_{L^2}\\
\le& Ch \|u(t_n)\|_{L^2} \|u(t_{n-1})\|_{H^2}\|\nabla v_h\|_{L^2}\\
\le& Ch t_n^{-1}\|\nabla v_h\|_{L^2}
\quad \mbox{for}\,\,\, v_h\in X_h \,\,\,\mbox{and} \,\,\, n \ge 2,
\end{align*}
where Lemma \ref{lemma:regularity-u} and $t_{n-1} \sim t_n$ are used for $n \ge 2$. 
For $n=1$ there holds 
\begin{align*}
|(\mathcal{F}_{h,3}^n, v_h)|
=& |( u(t_1),(P_{X_h}u^0 - u^0) \cdot \nabla v_h )|\\
\le&
C\|u(t_1)\|_{L^\infty} \|P_{X_h}u^0 - u^0\|_{L^2}\|\nabla v_h\|_{L^2}\\
\le& 
C\|u(t_1)\|_{L^2}^{\frac12} \|u(t_1)\|_{H^2}^{\frac12}
\|u^0\|_{L^2} 
\|\nabla v_h\|_{L^2}\\
\le& C t_1^{-\frac 12}\|\nabla v_h\|_{L^2} \quad\mbox{(here Lemma \ref{lemma:regularity-u} is used)} \\
\le& C \tau_1 t_1^{-\frac 32}\|\nabla v_h\|_{L^2} \quad\mbox{for}\,\,\,v_h\in X_h\,\,\,\mbox{and}\,\,\, n=1.
\end{align*}
Collecting the above estimates of $\mathcal{F}_{h,j}^n$, $j=1,2,3$, for $n\ge 2$ and $n=1$, we obtain \eqref{con_err-F}. 
%
%
\hfill\end{proof}
    %

%
%

\subsection{Error estimates in a sufficiently small time interval $[0,T_*]$} \label{section:duality}
$\,$

By subtracting \eqref{semi-Euler} from \eqref{FEM-int}, we obtain the following equation for the error function $e_h^n =\hat u_h^n - u_h^n$: 
\begin{align}\label{Error_Eq_FD}
\frac{e_h^n - e_h^{n-1}}{\tau_n}
- A_h e_h^n 
+ P_{X_h}(\hat u_h^{n-1}\cdot\nabla \hat u_h^n-u_h^{n-1}\cdot\nabla u_h^n)
= \mathcal{E}^n +  \mathcal{F}_h^n. 
\end{align}

We first estimate $\sum_{n=1}^m \tau_n \|e_h^n\|_{L^2}^2$ by a duality argument. To this end, we denote by $\phi_h^n\in X_h$ the solution of the backward problem 
\begin{align}\label{dual-eq-phih}
\left\{
\begin{aligned}
-\frac{\phi_h^n - \phi_h^{n-1}}{\tau_{n-1}} - A_h \phi_h^{n-1} 
&= e_h^{n-1} &&\mbox{for}\,\,\, n= 2,\dots,m+1, \\ 
\phi_h^{m+1}&=0  ,
\end{aligned}
\right.
\end{align}
which satisfies the following standard energy estimate: 
\begin{align}\label{phih-sta}
\max_{1\le n\le m}
\|\phi_h^n\|_{H^1}^2
+ \sum_{n=1}^m \tau_n \|A_h\phi_h^n\|_{L^2}^2
\le
C\sum_{n=1}^m \tau_n \|e_h^n\|_{L^2}^2.
\end{align}
This estimate can be obtained from testing \eqref{dual-eq-phih} by $-A_h\phi_h^{n-1}$ and summing up the results for $n=2,\dots,m+1$. 

Testing \eqref{dual-eq-phih} by $\tau_{n-1} e_h^{n-1}$ and summing up the results for $n=2,\dots,m+1$, and using discrete integration by parts in time (with $e_h^0=\phi_h^{m+1}=0$), we obtain 
\begin{align*}
\sum_{n=2}^{m+1} \tau_{n-1} \|e_h^{n-1}\|_{L^2}^2 
=& \sum_{n=2}^{m+1} \tau_{n-1} \Big ( e_h^{n-1}, - \frac{\phi_h^n - \phi_h^{n-1}}{\tau_{n-1}} - A_h \phi_h^{n-1} \Big ) \\
=& \sum_{n=1}^{m} ( e_h^n, \phi_h^n)-  \sum_{n=1}^m (e_h^{n-1}, \phi_h^n)  
+  \sum_{n=1}^{m} \tau_n ( \nabla e_h^n, \nabla \phi_h^n)  \\
=&  \sum_{n=1}^m \tau_n  
\Big(\frac{e_h^n - e_h^{n-1}}{\tau_n} - A_h e_h^n , \phi_h^n \Big )
\end{align*}
By substituting \eqref{Error_Eq_FD} into the inequality above, we obtain
\begin{align}\label{L2L2-eh-n-1}
&\sum_{n=2}^{m+1} \tau_{n-1} \|e_h^{n-1}\|_{L^2}^2 \notag \\
&= - \sum_{n=1}^m \tau_n  (\hat u_h^{n-1}\cdot\nabla \hat u_h^{n} - u_h^{n-1}\cdot\nabla u_h^n,\phi_h^n) 
+\sum_{n=1}^m \tau_n   ( \mathcal{E}^n, \phi_h^n ) 
+\sum_{n=1}^m \tau_n   ( \mathcal{F}_h^n, \phi_h^n )\notag \\
&= - \sum_{n=1}^m \tau_n  (e_h^{n-1}\cdot\nabla \hat u_h^{n} + u_h^{n-1}\cdot\nabla e_h^n , \phi_h^n ) 
+\sum_{n=1}^m \tau_n   ( \mathcal{E}^n, \phi_h^n ) 
+\sum_{n=1}^m \tau_n   ( \mathcal{F}_h^n, \phi_h^n )   \notag \\
&= \sum_{n=1}^m \tau_n  \Big[ ( \hat u_h^{n} , e_h^{n-1} \cdot\nabla \phi_h^n )  
+ ( e_h^n, u_h^{n-1} \cdot\nabla  \phi_h^n ) \Big] 
+\sum_{n=1}^m \tau_n   ( \mathcal{E}^n, \phi_h^n ) 
+ \sum_{n=1}^m \tau_n   ( \mathcal{F}_h^n, \phi_h^n ) ,
\end{align}
where we have used integration by parts in deriving the last equality. 
The last term on the right-hand side of the above equation can be estimated as follows: 
In the case of $n \ge 2$ we use the decomposition $\mathcal{F}_{h}^n=\mathcal{F}_{h, 1}^n+\mathcal{F}_{h, 2}^n+\mathcal{F}_{h, 3}^n$ with 
\begin{align*}
|(\mathcal{F}_{h, 1}^n, \phi_h^n)|
=&|( P_{X_h} u(t_n) - R_{X_h} u(t_n), A_h \phi_h^n )|\\
\le&\|P_{X_h} u(t_n) - R_{X_h} u(t_n)\|_{L^2}\| A_h \phi_h^n\|_{L^2}\\
\le&Ch \|u(t_n)\|_{H^1} \|A_h \phi_h^n\|_{L^2} 
\quad \mbox{(here \eqref{Error-PXh} and \eqref{Error-Stokes-Ritz} are used)} \\
|(\mathcal{F}_{h, 2}^n, \phi_h^n)|
=& |(P_{X_h} u(t_n)-u(t_n),P_{X_h} u(t_{n-1}) \cdot \nabla \phi_h^n )|  \\
\le&C \|P_{X_h} u(t_n)-u(t_n)\|_{L^2} \|P_{X_h} u(t_{n-1})\|_{L^4} \|\nabla \phi_h^n\|_{L^4}\\
\le&C\|P_{X_h} u(t_{n})-u(t_n)\|_{L^2} \|P_{X_h} u(t_{n-1})\|_{L^2}^{\frac 12} \| P_{X_h} u(t_{n-1})\|_{H^1}^{\frac 12}  \|\nabla  \phi_h^n\|_{L^2}^{\frac 12} \|A_h \phi_h^n\|_{L^2}^{\frac 12}\\
&\hspace{174pt} \mbox{(here \eqref{L4-L2-H1} and Lemma \ref{Lemma:discrete_W14} are used)} \\
\le&C
h \|u(t_n)\|_{H^1} \|u(t_{n-1})\|_{L^2}^{\frac 12}  \|u(t_{n-1})\|_{H^1}^{\frac 12}  \|\nabla  \phi_h^n\|_{L^2}^{\frac 12} \|A_h \phi_h^n\|_{L^2}^{\frac 12}
\quad\mbox{(here \eqref{H1-stability-PXh} is used)}  \\
|(\mathcal{F}_{h, 3}^n, \phi_h^n)|
=&|(u(t_n),(P_{X_h} u(t_{n-1}) - u(t_{n-1})) \cdot \nabla \phi_h^n)|\\
\le&C\|u(t_n)\|_{L^4} \|P_{X_h} u(t_{n-1}) - u(t_{n-1})\|_{L^2}\|\nabla \phi_h^n\|_{L^4}\\
\le& C\|u(t_n)\|_{L^2}^{\frac12} \|u(t_n)\|_{H^1}^{\frac12} h\|u(t_{n-1})\|_{H^1}
\|\nabla  \phi_h^n\|_{L^2}^{\frac 12} \|A_h \phi_h^n\|_{L^2}^{\frac 12} . 
\end{align*}
Since $\|u(t_n)\|_{L^2}\le C$, the three estimates above imply that
\begin{align}\label{F-n-bigger-2}
|( \mathcal{F}_h^n, \phi_h^n )| 
\le& Ch  \|u(t_n)\|_{H^1} \|A_h\phi_h^n\|_{L^2}  \notag \\
&+ Ch  (\|u(t_n)\|_{H^1}^{\frac32}+ \|u(t_{n-1})\|_{H^1}^{\frac32} )\|\nabla  \phi_h^n\|_{L^2}^{\frac 12}\|A_h\phi_h^n\|_{L^2}^{\frac 12} \notag \\
\le& \epsilon \|A_h\phi_h^n\|_{L^2}^2 
+ C\epsilon^{-1} h^2 (\|u(t_n)\|_{H^1}^2 + \|u(t_{n-1})\|_{H^1}^2 ) \notag \\
&+ \epsilon (\|u(t_n)\|_{H^1} + \|u(t_{n-1})\|_{H^1})\|\nabla  \phi_h^n\|_{L^2}\|A_h\phi_h^n\|_{L^2} \quad\mbox{for}\,\,\, n\ge 2 , 
\end{align}
where $\epsilon$ can be an arbitrary positive constant (arising from using Young's inequality). 
In the case $n = 1$, Lemma~\ref{LM:Consistency-Euler} implies that 
\begin{align}\label{F-n--1}
\tau_1 |(\mathcal{F}_h^1, \phi_h^1)|
\le C\tau_1^{\frac12} \|\nabla  \phi_h^1\|_{L^2}
\le C\epsilon^{-1} \tau_1 +\epsilon\|\nabla  \phi_h^1\|_{L^2}^2 
\le C\epsilon^{-1} t_m^{-1} \tau_m^2 +\epsilon\|\nabla  \phi_h^1\|_{L^2}^2 , 
\end{align}
where the last inequality is due to the stepsize choice in \eqref{stepsize}, which implies that 
\begin{align}\label{tau1-estimate}
\tau_1
\le C \tau^{\frac{1}{1-\alpha}} 
\le C \tau^{\frac{1}{1-\alpha}-2 } t_m^{1-2\alpha} (t_m^{2\alpha-1} \tau^2 )
\le C \tau^{\frac{1}{1-\alpha}-2 } t_m^{1-2\alpha} (t_m^{-1} \tau_m^2 )
&\le C \tau^{(\frac{1}{1-\alpha}-2)\alpha} \, t_m^{-1} \tau_m^2 \notag \\ 
&\le C t_m^{-1} \tau_m^2 \,\,\,\,\,\mbox{for}\,\,\,\alpha\in(\mbox{$\frac12$},1) .
\end{align}
By summing $\tau_n|( \mathcal{F}_h^n, \phi_h^n )| $ for $n=1,\dots,m$, and using the estimates in \eqref{F-n-bigger-2}--\eqref{F-n--1}, we obtain 
\begin{align}\label{sum-Fhn-phihn}
\sum_{n=1}^m \tau_n   |( \mathcal{F}_h^n, \phi_h^n )| 
=&\tau_1 |(\mathcal{F}_h^1, \phi_h^1)| + \sum_{n=2}^m \tau_n   |( \mathcal{F}_h^n, \phi_h^n )| \notag \\
\le&\tau_1 |(\mathcal{F}_h^1, \phi_h^1)| 
+  \epsilon \sum_{n=2}^m  \tau_n  \|A_h\phi_h^n\|_{L^2}^2  
+ C\epsilon^{-1} h^2 \sum_{n=2}^m \tau_n (\|u(t_{n-1})\|_{H^1}^2 + \|u(t_n)\|_{H^1}^2) \notag \\
&+ \epsilon \sum_{n=2}^m \tau_n  (\|u(t_n)\|_{H^1} + \|u(t_{n-1})\|_{H^1})\|\nabla  \phi_h^n\|_{L^2}\|A_h\phi_h^n\|_{L^2} \notag \\
\le& C\epsilon^{-1} t_m^{-1} \tau_m^2  +\epsilon\|\nabla  \phi_h^1\|_{L^2}^2 
+\epsilon\sum_{n=1}^m \tau_n \|A_h\phi_h^n\|_{L^2}^2 
+ C\epsilon^{-1} h^2 \sum_{n=1}^m \tau_n \|u(t_n)\|_{H^1}^2 \notag  \\
&+ \epsilon \bigg(\sum_{n=1}^m \tau_n \|u(t_n)\|_{H^1}^2 \bigg)^{\frac 12}
\max_{1\le n\le m} \|\phi_h^n\|_{H^1} \bigg(\sum_{n=1}^m \tau_n \|A_h\phi_h^n\|_{L^2}^2\bigg)^{\frac12} 
\notag \\
\le& 
C\epsilon^{-1} t_m^{-1} \tau_m^2  
+ C\epsilon \sum_{n=1}^m \tau_n \|e_h^{n}\|_{L^2}^2 
+ C\epsilon^{-1} h^2  \sum_{n=1}^m \tau_n \|u(t_n)\|_{H^1}^2 \notag \\
&
+ C\epsilon \bigg( \sum_{n=1}^m \tau_n \|u(t_n)\|_{H^1}^2\bigg)^{\frac 12}
\bigg(\sum_{n=1}^m \tau_n \|e_h^{n}\|_{L^2}^2\bigg)
\quad\mbox{(here \eqref{phih-sta} is used)} . 
\end{align}
By using \eqref{con_err-E} we have 
\begin{align}\label{sum-Ehn-phihn}
\sum_{n=1}^m \tau_n   |( \mathcal{E}_h^n, \phi_h^n )| 
&\le
C\bigg( \sum_{n=1}^m \tau_n^2 t_n^{-\frac 32} \bigg)
\max_{1\le n\le m}
\|\phi_h^n\|_{H^1} \notag\\
&\le
C\bigg( \sum_{n=1}^m \tau_n \tau t_n^{\alpha-\frac 32} \bigg)
\bigg(\sum_{n=1}^m \tau_n \|e_h^{n}\|_{L^2}^2\bigg)^{\frac12} \notag\\
&\le
Ct_m^{\alpha-\frac12} \tau 
\bigg(\sum_{n=1}^m \tau_n \|e_h^{n}\|_{L^2}^2\bigg)^{\frac12}
\,\quad\mbox{since $\alpha>\frac12$} \notag\\
&\le
Ct_m^{-\frac12} \tau_m 
\bigg(\sum_{n=1}^m \tau_n \|e_h^{n}\|_{L^2}^2\bigg)^{\frac12}
\quad\mbox{here we used \eqref{stepsize}} \notag\\
&\le
C \epsilon^{-1}t_m^{-1} \tau_m^2 + \epsilon\sum_{n=1}^m \tau_n \|e_h^{n}\|_{L^2}^2 .
\end{align}
Note that 
\begin{align}
&\sum_{n=1}^m \tau_n \|\hat u_h^{n}\|_{L^4}^4
\le
C\sum_{n=1}^m \tau_n \|\hat u_h^{n}\|_{L^2}^2 \|\hat u_h^{n}\|_{H^1}^2 
\le
C\sum_{n=1}^m \tau_n  \|\hat u_h^{n}\|_{H^1}^2 
\le 
C\sum_{n=1}^m \tau_n  \| u(t_n)\|_{H^1}^2 , \\ 
&\sum_{n=1}^m \tau_n \| u_h^{n-1}\|_{L^4}^4
\le
C\sum_{n=1}^m \tau_n \| u_h^{n-1}\|_{L^2}^2 \|u_h^{n-1}\|_{H^1}^2 
\le
C\sum_{n=1}^m \tau_n  \|u_h^{n-1}\|_{H^1}^2  ,\\ 
&\sum_{n=1}^m \tau_n \|\nabla\phi_h^n\|_{L^4}^4
\le
C\sum_{n=1}^m \tau_n \|\phi_h^n\|_{H^1}^2 \|A_h\phi_h^n\|_{L^2}^2
\le
C\max_{1\le n\le m}\|\phi_h^n\|_{H^1}^2 
\sum_{n=1}^m \tau_n \|A_h\phi_h^n\|_{L^2}^2 \notag \\
&\hspace{196pt} \le C\bigg(\sum_{n=1}^m \tau_n \|e_h^{n}\|_{L^2}^2\bigg)^2,
\label{phi_h_L4L4}
\end{align}
which are consequences of  \eqref{L4-L2-H1} and Lemma \ref{Lemma:discrete_W14}. 
Substituting \eqref{sum-Fhn-phihn}--\eqref{sum-Ehn-phihn} into \eqref{L2L2-eh-n-1} and using the three estimates above, we obtain 
\begin{align}\label{L2L2-eh-n-2}
&\sum_{n=1}^m \tau_n \|e_h^n\|_{L^2}^2 \notag \\
&\le
C\bigg(\sum_{n=1}^m \tau_n \|\hat u_h^n \|_{L^4}^4\bigg)^{\frac14} 
\bigg(\sum_{n=1}^m \tau_n \|e_h^{n-1}\|_{L^2}^2\bigg)^{\frac12} 
\bigg(\sum_{n=1}^m \tau_n \|\nabla \phi_h^n\|_{L^4}^4\bigg)^{\frac14} \notag \\
&\quad\, 
+ C\bigg(\sum_{n=2}^m \tau_n \|u_h^{n-1}\|_{L^4}^4\bigg)^{\frac14} 
\bigg( \sum_{n=2}^m\tau_n \|e_h^{n}\|_{L^2}^2\bigg)^{\frac12} 
\bigg( \sum_{n=2}^m \tau_n \|\nabla \phi_h^n\|_{L^4}^4\bigg)^{\frac14} \notag \\
&\quad\,  + \tau_1|( e_h^1, u_h^0 \cdot\nabla  \phi_h^1 )| 
+\sum_{n=1}^m \tau_n   |( \mathcal{E}_h^n, \phi_h^n )|
+\sum_{n=1}^m \tau_n   |( \mathcal{F}_h^n, \phi_h^n )| \notag \\
&\le
C
\bigg(\sum_{n=1}^m \tau_n  \|\nabla u(t_n)\|_{L^2}^2
+\sum_{n=2}^m \tau_n  \|\nabla u_h^{n-1}\|_{L^2}^2\bigg)^{\frac14} \bigg(\sum_{n=1}^m \tau_n \|e_h^{n}\|_{L^2}^2\bigg)^{\frac12}  \bigg(\sum_{n=1}^m \tau_n \|\nabla \phi_h^n\|_{L^4}^4\bigg)^{\frac14} \notag \\
&\quad\,  + \tau_1|( e_h^1, u_h^0 \cdot\nabla  \phi_h^1 )| 
+\sum_{n=1}^m \tau_n   |( \mathcal{E}_h^n, \phi_h^n )|
+\sum_{n=1}^m \tau_n   |( \mathcal{F}_h^n, \phi_h^n )| \notag \\
&\le
C
\bigg(\sum_{n=1}^m \tau_n  \| u(t_n)\|_{H^1}^2
+\sum_{n=2}^m \tau_n  \| u_h^{n-1}\|_{H^1}^2\bigg)^{\frac14} \bigg(\sum_{n=1}^m \tau_n \|e_h^{n}\|_{L^2}^2\bigg) \qquad \mbox{(here \eqref{phi_h_L4L4} is used)} \notag \\
&\quad\,
+ \tau_1|( e_h^1, u_h^0 \cdot\nabla  \phi_h^1 )|  
+ C\epsilon^{-1} t_m^{-1} \tau_m^2 
+ C\epsilon \sum_{n=1}^m \tau_n \|e_h^{n}\|_{L^2}^2 
+ C\epsilon^{-1} h^2  \sum_{n=1}^m \tau_n \|u(t_n)\|_{H^1}^2 \notag \\
&\quad\,
+ C\epsilon \bigg( \sum_{n=1}^m \tau_n \|u(t_n)\|_{H^1}^2\bigg)^{\frac 12}
\bigg(\sum_{n=1}^m \tau_n \|e_h^{n}\|_{L^2}^2\bigg) 
\end{align}

The remaining term $\tau_1|(e_h^1, u_h^0 \cdot\nabla  \phi_h^1)|$ in \eqref{L2L2-eh-n-2} can be estimated by using the basic energy estimate:
\begin{align}\label{L2-H1-eh1}
\|e_h^1\|_{L^2}^2 +  \tau_1 \|\nabla e_h^1\|_{L^2}^2  \le C , 
\end{align}
which is a combination of \eqref{FD-basic-energy} and the regularity estimate (see Lemma \ref{lemma:regularity-u}) 
$$
\|u(t_1)\|_{L^2}^2 +  \tau_1 \|\nabla u(t_1)\|_{L^2}^2  \le C
$$ 
through the triangle inequality. By using \eqref{L2-H1-eh1} we have 
\begin{align}
\tau_1|(e_h^1, u_h^0 \cdot\nabla  \phi_h^1 )|
\le& C\tau_1\|e_h^1\|_{L^4}\|u_h^0\|_{L^2}\|\nabla  \phi_h^1\|_{L^4} \notag\\
\le& C\tau_1\|e_h^1\|_{L^2}^{\frac12} \|\nabla e_h^1 \|_{L^2}^{\frac12} \|u_h^0\|_{L^2} \|\nabla  \phi_h^1\|_{L^2}^{\frac12}\|A_h \phi_h^1\|_{L^2}^{\frac12}
\notag\\
\le& C \tau_1^{\frac 34}\|e_h^1\|_{L^2}^{\frac12} \|\nabla \phi_h^1\|_{L^2}^{\frac12}\|A_h \phi_h^1\|_{L^2}^{\frac12} 
\quad\mbox{(using \eqref{L2-H1-eh1} and $\|u_h^0\|_{L^2}\le C$)}
\notag\\
\le&C\epsilon^{-1} \tau_1\|e_h^1\|_{L^2} + \epsilon \tau_1^{\frac 12}\|\phi_h^1\|_{H^1}\|A_h \phi_h^1\|_{L^2}
\notag\\
\le&C\epsilon^{-3} \tau_1   + \epsilon \sum_{n=1}^m \tau_n \|e_h^n\|_{L^2}^2 + \frac 12\epsilon (\|\phi_h^1\|_{H^1}^2 + \tau_1\|A_h \phi_h^1\|_{L^2}^2) \notag\\
\le&C\epsilon^{-3} t_m^{-1} \tau_m^2   + C \epsilon \sum_{n=1}^m \tau_n \|e_h^n\|_{L^2}^2 ,
\end{align}
where we have used \eqref{phih-sta} and \eqref{tau1-estimate} in deriving the last inequality. Substituting the last inequality into \eqref{L2L2-eh-n-2} and choosing a sufficiently small constant $\epsilon$ (so that the term $C\epsilon\sum_{n=1}^m \tau_n \|e_h^n\|_{L^2}^2$ can be absorbed by the left-hand side), we obtain that 
\begin{align}\label{L2-L2-ehn}
\sum_{n=1}^m \tau_n \|e_h^n\|_{L^2}^2 
&\le
C
\bigg[ \bigg(\sum_{n=1}^m \tau_n  (\|\nabla u(t_n)\|_{L^2}^2 + \|\nabla u_h^n\|_{L^2}^2) \bigg)^{\frac14} 
+\bigg( \sum_{n=1}^m \tau_n \|u(t_n)\|_{H^1}^2\bigg)^{\frac 12}\bigg]
\bigg(\sum_{n=1}^m \tau_n \|e_h^{n}\|_{L^2}^2\bigg) \notag\\
&\quad\, + Ct_m^{-1} \tau_m^2  + Ch^2 \sum_{n=1}^m \tau_n \|u(t_n)\|_{H^1}^2 
\end{align}

The first and last terms on the right-hand side of \eqref{L2-L2-ehn} can be dealt with using the following lemma. 
\begin{lemma}\label{Lemma:L2L2uhn}
For any given $u^0\in \dot L^2$ the following result holds: 
$$
\sum_{n=1}^N \tau_n (\|\nabla u(t_n)\|_{L^2}^2 + \|\nabla u_h^n\|_{L^2}^2) \le C . 
$$
Furthermore, for any $\varepsilon>0$ there exists positive constants $T_\varepsilon$, $h_\varepsilon$ and $\tau_\varepsilon$ {\rm(}depending on $u^0$, but independent of $\tau$ and $h)$ such that for $h\le h_\varepsilon$ and $\tau\le \tau_\varepsilon$ the following result holds: 
$$
\sum_{n=1}^m \tau_n (\|\nabla u(t_n)\|_{L^2}^2 + \|\nabla u_h^n\|_{L^2}^2) \le \varepsilon
\quad
\forall\, t_m \in (0,T_\varepsilon] . 
$$
\end{lemma}

The proof of Lemma \ref{Lemma:L2L2uhn} is deferred to Section \ref{sec:AppendixA}. 

By using Lemma \ref{Lemma:L2L2uhn}, there exist constants $T_*$, $h_*$ and $\tau_*$ such that for $h\le h_*$, $\tau\le \tau_*$ and $t_m \le T_*$ the quantity 
$$
\sum_{n=1}^m \tau_n (\|\nabla u(t_n)\|_{L^2}^2 + \|\nabla u_h^n\|_{L^2}^2)
$$
is sufficiently small so that the first term on the right-hand side of \eqref{L2-L2-ehn} can be absorbed by the left hand side, and the last term on the right-hand side of \eqref{L2-L2-ehn} is bounded by $Ch^2$. Then we obtain 
\begin{align}\label{L2-L2-ehn-f}
&\sum_{n=1}^m \tau_n \|e_h^n\|_{L^2}^2 
\le C(t_m^{-1} \tau_m^2  + h^2) \quad\mbox{for}\,\,\, t_m \in (0,T_*]. 
\end{align}
Hence, if we consider the problem in the time interval $[0,T_*]$, we obtain an error estimate \eqref{L2-L2-ehn-f} in the discrete $L^2(0,T_*;L^2)$ norm. 

The error bound in \eqref{L2-L2-ehn-f} can be furthermore improved to an $L^2$ norm at a fixed time. To this end, we denote by $\chi(t)$ the nonnegative smooth cut-off function such that 
\begin{align}\label{def-chi-t-1}
\chi(t)
=\left\{
\begin{aligned}
&0 &&\mbox{for}\,\,\,t\in(0,t_m/4],\\
&1 &&\mbox{for}\,\,\,t\in[t_m/2,\infty) ,
\end{aligned}
\right.
\quad\mbox{and}\quad
|\partial_t^k\chi(t)| \le Ct_m^{-k}\,\,\, \mbox{for}\,\,\, k=0,1,2,\dots 
\end{align}
which satisfies 
$|\partial_t\chi(t)|\le Ct_m^{-1}$. 
Then, testing \eqref{Error_Eq_FD} by $\chi(t_n) e_h^n$, we obtain   
\begin{align} \label{point-L2-1}
&\frac{\chi(t_n) \| e_h^n \|_{L^2}^2 - \chi(t_{n-1}) \|e_h^{n-1}\|_{L^2}^2  
+\chi(t_n) \|e_h^n - e_h^{n-1} \|_{L^2}^2 }{2 \tau_n}
+\chi(t_n)  \|\nabla e_h^n\|_{L^2}^2 \notag \\ 
&= 
\chi(t_n) (\hat u_h^{n},e_h^{n-1}\cdot\nabla e_h^n) 
+ \chi(t_n) (\mathcal{E}^n,e_h^n) 
+ \chi(t_n) (\mathcal{F}_h^n,e_h^n) 
+ \frac{(\chi(t_n) -\chi(t_{n-1}))\| e_h^{n-1} \|_{L^2}^2}{2 \tau_n} \notag \\
&\le 
\chi(t_n) \|\hat u_h^{n}\|_{L^4}\|e_h^{n-1}\|_{L^4}\|\nabla e_h^n\|_{L^2}
\notag \\
&\quad\, 
+   C\chi(t_n)(\tau_n t_n^{-\frac 32} + h t_n^{-1}) \|\nabla e_h^n\|_{L^2} 
+ \max_{t\in[t_{n-1},t_n]} |\partial_t\chi(t)| \| e_h^{n-1} \|_{L^2}^2
\quad\mbox{(Lemma \ref{LM:Consistency-Euler} is used)} \notag \\
&\le 
\chi(t_n) \|\hat u_h^{n}\|_{L^2}^{\frac12}\|\hat u_h^{n}\|_{H^1}^{\frac12}
\|e_h^{n-1}\|_{L^2}^{\frac12}\|\nabla e_h^{n-1}\|_{L^2}^{\frac12}
\|\nabla e_h^n\|_{L^2} \notag \\
&\quad\, 
+  C\chi(t_n)(\tau_n t_n^{-\frac 32} + h t_n^{-1}) \|\nabla e_h^n\|_{L^2} 
+ Ct_m^{-1}  \| e_h^{n-1} \|_{L^2}^2 \notag \\
&\le 
C\chi(t_n) \|\hat u_h^{n}\|_{H^1}^2 \|e_h^{n-1}\|_{L^2}^2
+ C\chi(t_n)(\tau_n^2 t_n^{-3} + h^2 t_n^{-2})
+ Ct_m^{-1}  \| e_h^{n-1} \|_{L^2}^2  \notag \\
&\quad\, 
+ \frac14 \chi(t_n) (\|\nabla e_h^{n-1}\|_{L^2}^2 + \|\nabla e_h^n\|_{L^2}^2) \notag \\
&\le 
C\|u(t_n)\|_{H^1}^2 \chi(t_n) \|e_h^{n-1}\|_{L^2}^2
+  C\chi(t_n)(\tau_n^2 t_n^{-3} + h^2 t_n^{-2}) 
+ Ct_m^{-1}  \| e_h^{n-1} \|_{L^2}^2  \notag \\
&\quad\, 
+ \frac14 \chi(t_n) (\|\nabla e_h^{n-1}\|_{L^2}^2 + \|\nabla e_h^n\|_{L^2}^2) , 
\end{align}
where in the last inequality we have used $\|\hat u_h^{n}\|_{H^1} = \| P_{X_h}u(t_n)\|_{H^1} \le C\|u(t_n)\|_{H^1}$ as a result of \eqref{H1-stability-PXh}. 
Absorbing the last term of \eqref{point-L2-1} by its left-hand side and applying the discrete version of Gronwall's inequality (cf. \cite[Lemma 5.1]{Heywood-Rannacher-1990}), we obtain 
\begin{align} 
&\max_{1\le n\le m} \chi(t_n) \| e_h^n \|_{L^2}^2 \notag \\ 
&\le 
\exp\Big(\sum_{n=1}^m\tau_n \|u(t_n)\|_{H^1}^2 \Big) 
\bigg(C \sum_{n=1}^m \chi(t_n) (\tau_n^3 t_n^{-3}+ \tau_n h^2t_n^{-2})
+  \sum_{n=1}^m\tau_n Ct_m^{-1}   \| e_h^{n-1} \|_{L^2}^2 \bigg) \notag \\
&\le 
\exp\big(\sum_{n=1}^m\tau_n \|u(t_n)\|_{H^1}^2 \big) 
\bigg( C  (\tau_m^2 t_m^{-3}+ h^2 t_m^{-2}) \sum_{n=1}^m \tau_n  1_{t_n\ge \frac{t_m}{4}} 
+ C t_m^{-1} \sum_{n=1}^m\tau_n \| e_h^{n-1} \|_{L^2}^2 \bigg) \notag \\
&\le C\tau_m^2 t_m^{-2} + Ch^2t_m^{-1} , 
\end{align}
where the last inequality uses Lemma \ref{Lemma:L2L2uhn} and \eqref{L2-L2-ehn-f}. 
Since this inequality holds for all $m\ge 1$ such that $t_m \in(0,T_*]$, it follows that 
\begin{align}\label{L2-Error-T*}
&\| e_h^n \|_{L^2} 
\le
C  (t_n^{-2} \tau_n^2  + t_n^{-1}  h^2) \quad\mbox{for}\,\,\, t_n \in (0,T_*].
\end{align}
This proves the desired error bound in a time interval $(0,T_*]$, where $T_*$ is a sufficiently small constant (depending on $u^0$ but independent of $\tau$ and $h$). In the next subsection, we extend the error estimate to the whole time interval $[0,T]$. 

\subsection{Error analysis in $[0,T]$} 
$\,$

Let $k$ be the maximal integer such that $t_k\in(0,T_*]$. When $\tau\le T_*/4$ there holds $t_k\ge T_*/2$ and therefore $t_k^{-1}\le C$. In this case, \eqref{L2-Error-T*} implies 
\begin{align}\label{Initial_Error_tm}
&\| e_h^k \|_{L^2} \le C(\tau_k + h) ,
\end{align}
and Lemma \ref{lemma:regularity-u} implies that 
\begin{align}\label{Ineq:discrete_u_t-T}
\|\partial_t^mu\|_{L^2}+ \|\partial_t^mu\|_{H^1} 
+ \|\partial_t^mu\|_{H^2}
\le C_m ,\,\,\, \forall\, t\in [t_k,T], \,\, m=0,1,\dots
\end{align}
Since $t_k^{-1}\le C$, the estimates in Lemma \ref{LM:Consistency-Euler} reduce to  
\begin{align*}
|(\mathcal{E}^n, v_h)| + |(\mathcal{F}_h^n, v_h)| \le C(\tau_n + h)\|\nabla v_h\|_{L^2} \quad\mbox{for}\,\,\, k+1\le n\le N, \quad v_h \in X_h. 
\end{align*}
Then, testing the error equation \eqref{Error_Eq_FD} by $e_h^n$, we obtain 
\begin{align*} 
&\frac{ \| e_h^n \|_{L^2}^2 - \|e_h^{n-1}\|_{L^2}^2 +\|e_h^n - e_h^{n-1} \|_{L^2}^2 }{2 \tau_n}
+ \|\nabla e_h^n\|_{L^2}^2 \notag \\ 
&= 
(\hat u_h^n,e_h^{n-1}\cdot\nabla e_h^n) 
+ (\mathcal{E}^n,e_h^n) +  (\mathcal{F}_h^n,e_h^n)\notag \\
&\le 
\|\hat u_h^n\|_{L^4}\|e_h^{n-1}\|_{L^4}\|\nabla e_h^n\|_{L^2}
+ C(\tau_n + h)\|\nabla e_h^n\|_{L^2} \notag \\
&\le 
C \|\hat u_h^n\|_{L^2}^{\frac12}\|\hat u_h^n\|_{H^1}^{\frac12}
\|e_h^{n-1}\|_{L^2}^{\frac12}\|\nabla e_h^{n-1}\|_{L^2}^{\frac12}
\|\nabla e_h^n\|_{L^2}  
+ C(\tau_n + h)\|\nabla e_h^n\|_{L^2} \notag \\
&\le 
C \|\hat u_h^n\|_{H^1}^2 \|e_h^{n-1}\|_{L^2}^2
+ C (\tau_n^2 + h^2)
+ \frac14 (\|\nabla e_h^{n-1}\|_{L^2}^2 + \|\nabla e_h^n\|_{L^2}^2) \notag \\
&\le 
C \|e_h^{n-1}\|_{L^2}^2
+ C (\tau_n^2 + h^2)
+ \frac14 (\|\nabla e_h^{n-1}\|_{L^2}^2 + \|\nabla e_h^n\|_{L^2}^2) 
\qquad\mbox{for}\,\,\,  n\ge k + 1 ,
\end{align*}
where we have used the regularity estimate 
$$
\|\hat u_h^n\|_{H^1}
=
\|P_{X_h} u(t_n)\|_{H^1} \le
\|u(t_n)\|_{H^1}\le C 
\quad\mbox{as a result of \eqref{H1-stability-PXh} and \eqref{Ineq:discrete_u_t-T}} .  
$$
By applying the discrete version of Gronwall's inequality, we obtain 
\begin{align}\label{H2-data-error}
\max_{k+1\le n\le N}\| e_h^n \|_{L^2} \le C\| e_h^k \|_{L^2}+C(\tau + h).  
\end{align}
which together with \eqref{Initial_Error_tm} yields the desired result of Theorem \ref{THM:FEM-Euler}. 
\hfill\endproof

In the proof of Theorem \ref{THM:FEM-Euler} we have used the key technical Lemma \ref{Lemma:L2L2uhn}, which is proved in the next subsection.

\subsection{Proof of Lemma \ref{Lemma:L2L2uhn}}\label{sec:AppendixA}

Lemma \ref{Lemma:L2L2uhn} is a combination of \eqref{FD-basic-energy} and the following two lemmas (Lemma \ref{Lemma:FD-small} and Lemma \ref{Lemma:FEM-small}).

\begin{lemma}\label{Lemma:FD-small}
Let $u^0\in \dot L^2$ be given. Then for any $\varepsilon>0$ there exist positive constants $T_\varepsilon$, $h_\varepsilon$ and $\tau_\varepsilon$ such that for $h\le h_\varepsilon$ and $\tau\le \tau_\varepsilon$ there holds 
\begin{align}\label{FD-L2H1-small-epsilon}
\sum_{n=1}^m \tau_n \|\nabla u_h^{n}\|_{L^2}^2 \le \varepsilon
\quad
\forall\, t_m\in(0,T_\varepsilon]  . 
\end{align}
The constants $T_\varepsilon$, $h_\varepsilon$ and $\tau_\varepsilon$  may depend on $u^0$ but are independent of $\tau$ and $h$. 
\end{lemma}

\begin{proof}
Let $u_{\tau,h}(t)$ be a piecewise linear function in time, defined by 
$$
u_{\tau,h}(t)=\frac{t_n-t}{\tau_n}u_h^{n-1}+\frac{t-t_{n-1}}{\tau_n}u_h^{n} \quad\mbox{for}\,\,\, t\in (t_{n-1},t_n] .
$$
We claim that 
\begin{align}
&\mbox{$u_{\tau,h}$ converges to the unique weak solution $u$ weakly in $L^2(0,T;\dot H^1)$}; 
\label{converg-utauh-0} \\
&\mbox{$u_{\tau,h}$ converges to $u$ strongly in $C([T_1,T];\dot L^2) $ for any fixed $T_1\in(0,T)$}; \label{converg-utauh-1} \\
&\mbox{$u_{\tau,h}$ converges to $u$ strongly in $C([0,T];\dot H^{-1})$} . 
\label{converg-utauh-2} 
\end{align}
The proof of \eqref{converg-utauh-0}--\eqref{converg-utauh-2} is presented in Appendix \ref{section:AppendixB}. 

In addition to \eqref{converg-utauh-0}--\eqref{converg-utauh-2}, we claim that the following result holds: 
\begin{align}\label{converg-CL2}
\max_{t\in[0,T]} \Big| \|u_{\tau,h}(\cdot,t)\|_{L^2} - \|u(\cdot,t)\|_{L^2} \Big| 
\rightarrow 0\,\,\,\mbox{as $\tau,h\rightarrow 0$} .
%
\end{align}
We prove \eqref{converg-CL2} by using the method of contradiction.
If \eqref{converg-CL2} does not hold then \eqref{converg-utauh-1} implies that there exists a sequence $t_j\rightarrow 0$ and $\tau_j,h_j\rightarrow 0$ such that 
\begin{align*}  
\Big| \|u_{\tau_j,h_j}(\cdot,t_j)\|_{L^2} - \|u(\cdot,t_j)\|_{L^2}  \Big| \ge \delta \,\,\, \mbox{for}\,\,\, j\ge 1. 
\end{align*}
Since  $u\in C([0,T];L^2)$ it follows that $\|u(\cdot,t_j)\|_{L^2}\rightarrow \|u^0\|_{L^2}$ and therefore 
\begin{align*}  
\Big| \|u_{\tau_j,h_j}(\cdot,t_j)\|_{L^2} - \|u^0\|_{L^2}  \Big| \ge \frac{\delta}{2} \,\,\, \mbox{for sufficiently large $j$} . 
\end{align*}
The energy inequality \eqref{FD-basic-energy} implies that $\|u_{\tau_j,h_j}(\cdot,t_j)\|_{L^2}\le \|u^0_{h_j}\|_{L^2} = \|P_{X_{h_j}}u^0\|_{L^2} \le \|u^0\|_{L^2}$, and therefore 
\begin{align} \label{utauhL2-u0L2}
\|u_{\tau_j,h_j}(\cdot,t_j)\|_{L^2} - \|u^0\|_{L^2}  \le -\frac{\delta}{2} \,\,\, \mbox{for sufficiently large $j$} . 
\end{align}
Since $u_{\tau_j,h_j}(\cdot,t_j)$ converges to $u^0$ in $\dot H^{-1}$ (as a result of  \eqref{converg-utauh-2}), and $\|u_{\tau_j,h_j}(\cdot,t_j)\|_{L^2}$ is uniformly bounded as $j\rightarrow\infty$, it follows that $u_{\tau_j,h_j}(\cdot,t_j)$ also converges to $u^0$ weakly in $L^2$ and therefore 
\begin{align} \label{utauhL2-u0L2-2}
\|u^0\|_{L^2} \le
\liminf_{j\rightarrow \infty} \|u_{\tau_j,h_j}(\cdot,t_j)\|_{L^2} . 
\end{align}
Substituting this into \eqref{utauhL2-u0L2} yields that 
$$
\|u^0\|_{L^2} \le \|u^0\|_{L^2} -\frac{\delta}{2} .
$$ 
The contradiction implies that \eqref{converg-CL2} holds. 

We use the standard energy equality for the numerical solution: 
\begin{align}\label{FD-energy-equality}
\frac{\|u_h^n\|_{L^2}^2-\|u_h^{n-1}\|_{L^2}^2}{2\tau_n}
+\frac{\tau_n}{2}\bigg\|\frac{u_h^n-u_h^{n-1}}{\tau_n}\bigg\|_{L^2}^2
+\|\nabla u_h^n\|_{L^2}^2
=0 , 
\end{align}
which can be obtained through testing \eqref{semi-Euler} by $u_h^n$. 
By summing up \eqref{FD-energy-equality} for $n=1,\dots,m$, we have 
\begin{align}
\sum_{n=1}^m\tau_n \|\nabla u_h^n\|_{L^2}^2
&\le \frac12\|u_h^0\|_{L^2}^2 - \frac12\|u_h^m\|_{L^2}^2 \notag \\
&=\frac12\|u_h^0\|_{L^2}^2 - \frac12\|u(t_m)\|_{L^2}^2
+\frac12(\|u(t_m)\|_{L^2}^2-\|u_h^m\|_{L^2}^2) \notag \\
&\le\frac12\|u^0\|_{L^2}^2 - \frac12\|u(t_m)\|_{L^2}^2
+\frac12(\|u(t_m)\|_{L^2}^2-\|u_h^m\|_{L^2}^2) \notag \\
&=\int_{0}^{t_m} \|\nabla u(t)\|_{L^2}^2 \d t 
+\frac12(\|u(t_m)\|_{L^2}^2-\|u_h^m\|_{L^2}^2) .
\end{align}
First, \eqref{converg-CL2} implies that there exist constants $\tau_\varepsilon$ and $h_\varepsilon$ such that when $\tau\le \tau_\varepsilon$ and $h\le h_\varepsilon$ the following inequality holds: 
$$\|u(t_m)\|_{L^2}^2-\|u_h^m\|_{L^2}^2<\varepsilon . $$ 
Second, $u\in L^2(0,T;H^1(\Omega)^2)$ implies that there exists a constant $T_\varepsilon$ such that 
$$\int_{0}^{T_\varepsilon} \|\nabla u(t)\|_{L^2}^2 \d t <\frac\varepsilon2 . 
$$ 
As a result, we have 
\begin{align}
\sum_{n=1}^m\tau_n \|\nabla u_h^n\|_{L^2}^2
&\le \int_{0}^{t_m} \|\nabla u(t)\|_{L^2}^2 \d t 
+\frac12(\|u(t_m)\|_{L^2}^2-\|u_h^m\|_{L^2}^2) 
<\varepsilon .
\end{align}
This proves the desired result of Lemma \ref{Lemma:FD-small}. 
\hfill\end{proof}

\begin{lemma}\label{Lemma:FEM-small}
Let $u^0\in \dot L^2$ be given. Then 
$ \sum\limits_{n=1}^N \tau_n \|\nabla u(t_n)\|_{L^2}^2 \le C. $
Moreover, for any $\varepsilon>0$ there exist positive constants $T_\varepsilon$, $h_\varepsilon$ and $\tau_\varepsilon$ such that for $h\le h_\varepsilon$ and $\tau\le \tau_\varepsilon$ there holds 
$$
\sum_{n=1}^m \tau_n \|\nabla u(t_n)\|_{L^2}^2 \le \varepsilon
\quad
\forall\,  t_m \in(0,T_\varepsilon] . 
$$
The constants $T_\varepsilon$, $h_\varepsilon$ and $\tau_\varepsilon$  may depend on $u^0$ but are independent of $\tau$ and $h$. 
\end{lemma}

\begin{proof}
By using the triangle inequality we have 
\begin{align}\label{discr-L2H1-hat-uhn}
\sum_{n=1}^m \tau_n \|\nabla u(t_n)\|_{L^2}^2
\le& 
C\sum_{n=1}^m \tau_n \bigg\| \frac{2}{\tau_{n}}\int_{t_{n-\frac12}}^{t_{n}}  \nabla (u(t_n)-u(t))\d t \bigg\|_{L^2}^2 + 
C\sum_{n=1}^m \tau_n \bigg\| \frac{2}{\tau_{n}}\int_{t_{n-\frac12}}^{t_{n}} \nabla u(t)\d t \bigg\|_{L^2}^2 \notag \\
\le& 
C\sum_{n=1}^m \tau_n \bigg\| \frac{2}{\tau_{n}}\int_{t_{n-\frac12}}^{t_{n}}   \int_t^{t_n} \nabla\partial_tu(s) \d s \d t \bigg\|_{L^2}^2 + 
C\sum_{n=1}^m  \int_{t_{n-\frac12}}^{t_{n}}  \|\nabla u(t)\|_{L^2}^2 \d t \notag \\
\le& 
C\sum_{n=1}^m \tau_n^2  \int_{t_{n-\frac12}}^{t_{n}}  \|\nabla\partial_tu(t)\|_{L^2}^2 \d t  + 
C \int_0^{t_m} \|\nabla u(t)\|_{L^2}^2 \d t \notag \\
\le& 
C \|t\,\partial_tu(t)\|_{L^2(0, t_m;H^1)}^2 + 
C\|u\|_{L^2(0, t_m;H^1)}^2. 
\end{align}
where we have used $\tau_n\le 2t$ for $t\in[t_{n-\frac12},t_n]$. 
From \eqref{Eq:u_energy_eq-L2} we see that $u \in L^2(0,T;H^1(\Omega)^2)$, which implies that 
\begin{align}\label{semi-L2H1-uhn-small}
\|u\|_{L^2(0,t_m;H^1)}^2 
<\varepsilon \quad\mbox{when}\,\,\, t_m\,\,\,\mbox{is smaller than some constant $T_\varepsilon$} .  
\end{align}
In view of \eqref{discr-L2H1-hat-uhn}, it suffices to prove the following result  
\begin{align}\label{L2H1-tdt-uh}
\|t\,\partial_tu \|_{L^2(0,s;H^1)}^2 \le C\|u\|_{L^2(0,s;H^1)}^2 . 
\end{align}
Then substituting \eqref{semi-L2H1-uhn-small}--\eqref{L2H1-tdt-uh} into \eqref{discr-L2H1-hat-uhn} yields the desired result of Lemma \ref{Lemma:FEM-small}. 

In order to prove \eqref{L2H1-tdt-uh}, we differentiate \eqref{pde_abstract} and consider the equation of $\partial_tu$, i.e., 
\begin{align}\label{AFEM-dt}
\partial_t^2u-A\partial_tu&=-P_X(\partial_tu\cdot\nabla u)-P_X(u\cdot\nabla \partial_tu)\quad\mbox{for}\,\,\, t\in(0,T] . 
\end{align}
Testing \eqref{AFEM-dt} by $t^2 \partial_tu$, we have 
\begin{align*} 
\frac12 t^2 \frac{\d}{\d t} \|\partial_t u\|_{L^2}^2 
+ t^2\|\nabla\partial_tu\|_{L^2}^2 
= & -(\partial_tu\cdot\nabla u, t^2\partial_t u) \\
= &\, t^2( u, \partial_tu\cdot\nabla\partial_t u) \quad\mbox{(integration by parts)} \\
\le & Ct^2\|u\|_{L^4}^2 \|\partial_tu\|_{L^4}^2 + \frac14 t^2\|\nabla\partial_tu\|_{L^2}^2 \\
\le & Ct^2\|u\|_{L^2}\|u\|_{H^1} \|\partial_tu\|_{L^2}\|\nabla \partial_tu\|_{L^2} + \frac14 t^2\|\nabla\partial_tu\|_{L^2}^2 \\
\le & \frac12 C_*\|\nabla u\|_{L^2}^2 t^2 \|\partial_tu\|_{L^2}^2 + \frac12 t^2\|\nabla\partial_tu\|_{L^2}^2,
\end{align*}
where we have used $\|u\|_{L^2}\le C$ in deriving the last inequality, as shown in \eqref{Eq:u_energy_eq-L2}. 
Since the second term on the right-hand side of the inequality above can be absorbed by its left-hand side, it follows from $t^2 \frac{\d}{\d t}\|\partial_t u\|_{L^2}^2=\frac{\d}{\d t} (t^2\| \partial_t u\|_{L^2}^2) - 2t\|\partial_t u\|_{L^2}^2$ that 
\begin{align}\label{dt-t2-dt-uh-L2}
\bigg(\frac{\d}{\d t} t^2\|\partial_tu\|_{L^2}^2\bigg) +t^2\|\nabla \partial_tu\|_{L^2}^2 
\le
C_*\|\nabla u\|_{L^2}^2 t^2\|\partial_tu\|_{L^2}^2 + 2C_*t\|\partial_tu\|_{L^2}^2 . 
\end{align}
Hence, it remains to estimate  $\int_0^{t_m} t\|\partial_tu\|_{L^2}^2 \d t$ (the last term of the inequality above). To this end, we test \eqref{pde_abstract} by $t \partial_t u$ and use \eqref{Eq:ut_energy_eq}. This yields that 
\begin{align*}
\frac{1}{2}t \frac{\d}{\d t}\|\nabla u\|_{L^2}^2
+ t\|\partial_tu\|_{L^2}^2 
=&-(u\cdot\nabla u, t \partial_t u)\\
\le&Ct^2\| u\|_{L^4}^2\|\partial_tu\|_{L^4}^2 
+\frac12\|\nabla u\|_{L^2}^2 \\
\le&Ct\| u\|_{L^2}\|\nabla u\|_{L^2}t\|\partial_tu\|_{L^2} \|\nabla \partial_tu\|_{L^2} 
+\frac12\|\nabla u\|_{L^2}^2 \\
\le&Ct\|\nabla u\|_{L^2} \|\nabla \partial_tu\|_{L^2} 
+\frac12\|\nabla u\|_{L^2}^2 
\quad\mbox{(here \eqref{Eq:ut_energy_eq} is used)} \\
\le&C\sigma^{-1} \|\nabla u\|_{L^2}^2 +\sigma t^2\|\nabla\partial_tu\|_{L^2}^2 
+\frac12\|\nabla u\|_{L^2}^2 ,
\end{align*}
where $\sigma\in(0,1)$ is a constant arising from Young's inequality and therefore can be arbitrarily small. By using the identity $\frac{t}{2} \frac{\d}{\d t}\|\nabla u\|_{L^2}^2=\frac{\d}{\d t} \bigl(\frac{t}{2}\|\nabla u\|_{L^2}^2\bigr) - \frac{1}{2}\|\nabla u\|_{L^2}^2$ we furthermore derive that 
\begin{align}\label{dt-t-grad-uh-L2}
t\|\partial_tu\|_{L^2}^2 + \frac{\d}{\d t}\bigg(\frac{t}{2} \|\nabla u\|_{L^2}^2\bigg)
\le
(C\sigma^{-1}+1)\|\nabla u\|_{L^2}^2 +\sigma t^2\|\nabla\partial_tu\|_{L^2}^2  . 
\end{align}

Combining the two estimates above, i.e., $2C_*\times$\eqref{dt-t-grad-uh-L2}$+$\eqref{dt-t2-dt-uh-L2}, we have 
\begin{align}
&\frac{\d}{\d t}\bigg( t^2\|\partial_tu\|_{L^2}^2
+ C_*t \|\nabla u\|_{L^2}^2\bigg) + t^2\|\nabla \partial_tu\|_{L^2}^2 
\notag \\
&\le
C_*\|\nabla u\|_{L^2}^2 t^2\|\partial_tu\|_{L^2}^2 
+(2C_*C\sigma^{-1} +2C_*)\|\nabla u\|_{L^2}^2 +2C_*\sigma t^2\|\nabla\partial_tu\|_{L^2}^2 .
\end{align}
By choosing $\sigma$ sufficiently small, the last term on the right-hand side can be absorbed by the left-hand side. Then we obtain  
\begin{align}
\frac{\d}{\d t} \big(t^2\|\partial_tu\|_{L^2}^2
+ C_*t \|\nabla u\|_{L^2}^2\big) + t^2\|\nabla \partial_tu\|_{L^2}^2 
\le
C\|\nabla u\|_{L^2}^2 t^2\|\partial_tu\|_{L^2}^2
+ C \|\nabla u\|_{L^2}^2 . 
\end{align}
By applying Gronwall's inequality  and using \eqref{Eq:u_energy_eq-L2} we obtain 
\begin{align}
&\max_{t\in[0,s]} \big(t^2\|\partial_tu\|_{L^2}^2
+ C_*t \|\nabla u\|_{L^2}^2\big) + \int_0^st^2\|\nabla \partial_tu\|_{L^2}^2 \d t \notag\\
&\le
\exp\Big(\int_0^s C\|\nabla u\|_{L^2}^2\d t\Big) \int_0^s C\|\nabla u\|_{L^2}^2 \d t \notag\\
&\le
C\int_0^s \|\nabla u\|_{L^2}^2 \d t  . 
\end{align}
This proves \eqref{L2H1-tdt-uh} and completes the proof of Lemma \ref{Lemma:FEM-small}. 
\hfill\end{proof}

\section{Numerical experiments}
In this section we present numerical examples to support the theoretical result in Theorem \ref{THM:FEM-Euler}. Both examples concern the incompressible NS problem
\begin{equation}
\left \{
\begin{aligned} 
\partial_t u + u\cdot\nabla u - \mu \varDelta u
+\nabla p &= 0 && \mbox{in}\,\,\,  \varOmega\times (0,T] ,\\
\nabla\cdot u&=0&&\mbox{in}\,\,\,  \varOmega\times (0,T] ,\\
u&=0 && \mbox{on}\,\,\, \partial\varOmega\times (0,T] ,\\
u&=u^0 && \mbox{at}\,\,\, \varOmega\times \{0\} , 
\end{aligned}
\right .
\end{equation}
in the unit square $\Omega=(0, 1)\times(0,1)$ with $T=0.1$ and $\mu = 0.05$. 
The Scott--Vogelius $(P_4, P_{3}^{-1})$ finite elements are used for spatial discretization; see \cite{scott-1985-norm}. This finite element space has the required properties (P1)--(P2) mentioned in Section \ref{sec:main_results}. 
All the computations are performed using the software package FEniCS (\url{https://fenicsproject.org}). 


\medskip

\begin{example}\label{example2}
{\upshape 
Let $w = \sin(\pi x)^{\epsilon+0.5} \sin(\pi y)^{\epsilon+0.5}$ with $\epsilon = 0.01$, and consider the initial value 
\begin{align*}
u^0=(u_{1}^0(x,y),u_{2}^0(x,y)) := (w_y, -w_x) , 
\end{align*}
which satisfies that
$$
u^0\in \dot L^2 \quad \mbox{but}\quad u^0\notin H^{\epsilon}(\Omega)^2 .
$$
We solve problem \eqref{pde} by the proposed method \eqref{fully-FEM-Euler} and compare the numerical solutions with the reference solution given by sufficiently small stepsize and mesh size.

}
\end{example}

The time discretization errors $\|u_h^N - u_{h,\rm ref}^N\|_{L^2(\Omega)}$ are presented in Table \ref{table_time_errors-2}, where the reference solution $u_{h,\rm ref}^N$ is chosen to be the numerical solution with maximal stepsize $\tau_{\rm ref}=1/1280$. We have used four sufficiently small spatial mesh sizes $h=2^{-4-j}, j = 0, 1, 2, 3$, to investigate the influence of spatial discretization on the temporal discretization errors $\|u_h^N - u_{h,\rm ref}^N\|_{L^2(\Omega)}$. From Table \ref{table_time_errors-2} we can see that the influence of spatial discretization is negligibly small in observing the first-order convergence in time, which is consistent with the result proved in Theorem \ref{THM:FEM-Euler}. 
\begin{table}[htp]\centering\small
\caption{Example \ref{example2}: Time discretization errors using variable stepsize with $\alpha=0.55$.}
\centering
\setlength{\tabcolsep}{4.0mm}{
\begin{tabular}{cccccc}
\toprule
$\tau$           &&         $h=1/16$&         $h=1/32$&         $h=1/64$&        $h=1/128$ \\[2pt]
\midrule
$1/40$           &&      3.8127E--02&      3.7780E--02&      3.7664E--02&      3.7624E--02 \\[2pt]
$1/80$           &&      1.5696E--02&      1.5545E--02&      1.5493E--02&      1.5475E--02 \\[2pt]
$1/160$          &&      7.6949E--03&      7.6225E--03&      7.5968E--03&      7.5879E--03 \\[2pt]
\midrule
Convergence rate && $O(\tau^{1.03})$& $O(\tau^{1.03})$& $O(\tau^{1.03})$& $O(\tau^{1.03})$ \\[2pt]
\bottomrule
\end{tabular}}
\label{table_time_errors-2}
\end{table}

The spatial discretization errors $\|u_h^N - u_{h,\rm ref}^N\|_{L^2(\Omega)}$ are presented in Table \ref{table_space_errors-2}, where the reference solution $u_{h,\rm ref}^N$ is chosen to be the numerical solution with mesh size $h_{\rm ref}=1/128$. We have chosen several sufficiently small time stepsizes $\tau = 2^{-3-j}/10, j = 0, 1, 2, 3$ to investigate the influence of temporal discretization on the spatial discretization errors $\|u_h^N - u_{h,\rm ref}^N\|_{L^2(\Omega)}$.  
From Table \ref{table_space_errors-2} we see that the influence of temporal discretization can be neglected compared with the spatial discretization errors, which are $O(h^{1.5})$ in the $L^2$ norm. This is half-order better than the result proved in Theorem \ref{THM:FEM-Euler}. The rigorous proof of this sharper convergence rate for $ L^2$ initial data still remains open. 

\begin{table}[htb]\centering\small
\caption{Example \ref{example2}: Spatial discretization errors using variable stepsize with $\alpha=0.55$. }
\centering
\setlength{\tabcolsep}{4.0mm}{
\begin{tabular}{cccccc}
\toprule
$h$              &&   $\tau=1/80$&  $\tau=1/160$&  $\tau=1/320$&  $\tau=1/640$ \\[2pt]
\midrule
$1/8$            &&   5.0365E--03&   4.7074E--03&   4.5093E--03&   4.4308E--03 \\[2pt]
$1/16$           &&   1.6844E--03&   1.5711E--03&   1.5019E--03&   1.4744E--03 \\[2pt]
$1/32$           &&   5.4146E--04&   5.0663E--04&   4.8451E--04&   4.7546E--04 \\[2pt]
\midrule
Convergence rate && $O(h^{1.64})$& $O(h^{1.63})$& $O(h^{1.63})$& $O(h^{1.63})$ \\[2pt]
\bottomrule
\end{tabular}}
\label{table_space_errors-2}
\end{table}

\medskip

\begin{example}\label{example1}
{\upshape 
In the second example, we consider an initial value $u^0 = P_X w$ with 
$$
w = (w_{1}(x,y),w_{2}(x,y)) = (y^{\epsilon - 0.5}, x^{\epsilon - 0.5}) \quad\mbox{with}\,\,\, \epsilon = 0.01 ,
$$ 
which is a function in $L^2(\Omega)^2$ but not in $H^{\epsilon}(\Omega)^2$. 
Since $P_X$ is the $L^2$-orthogonal projection from $L^2(\varOmega)^2$ onto $\dot L^2$, it follows that $u^0\in \dot L^2$. But the analytical expression of $u^0$ is unknown. 
We solve problem \eqref{pde} by the proposed method \eqref{fully-FEM-Euler} with $u_h^0=P_{X_h}u^0=P_{X_h}w$, which can be computed from \eqref{FEM-u0h} with $u^0$ replaced by $w$ therein. Then we compare the numerical solutions with a reference solution given by sufficiently small mesh size.

}
\end{example}

The temporal discretization errors $\|u_h^N - u_{h,\rm ref}^N\|_{L^2(\Omega)}$ are presented in Table \ref{table_time_errors}, where the reference solution $u_{h,\rm ref}^N$ is chosen to be the numerical solution with maximal stepsize $\tau_{\rm ref}=1/1280$, and we have used several sufficiently small spatial mesh sizes $h=2^{-4-j}, j = 0, 1, 2, 3$ to investigate the spatial discretization errors and to guarantees that the influence of spatial discretization error is negligibly small in observing the temporal convergence rates. 
From Table \ref{table_time_errors} we see that the temporal discretization errors are about $O(\tau)$, which is consistent with the result proved in Theorem \ref{THM:FEM-Euler}. 
\begin{table}[htp]\centering\small
\caption{Example \ref{example1}: Time discretization errors using variable stepsize with $\alpha=0.55$. }
\centering
\setlength{\tabcolsep}{4.0mm}{
\begin{tabular}{cccccc}
\toprule
$\tau$           &&         $h=1/16$&         $h=1/32$&         $h=1/64$&        $h=1/128$ \\[2pt]
\midrule
$1/40$           &&      4.3461E--03&      4.0663E--03&      3.9460E--03&      3.9014E--03 \\[2pt]
$1/80$           &&      1.8754E--03&      1.7471E--03&      1.6919E--03&      1.6711E--03 \\[2pt]
$1/160$          &&      9.1079E--04&      8.4811E--04&      8.1938E--04&      8.0861E--04 \\[2pt]
\midrule
Convergence rate && $O(\tau^{1.04})$& $O(\tau^{1.04})$& $O(\tau^{1.05})$& $O(\tau^{1.05})$ \\[2pt]
\bottomrule
\end{tabular}}
\label{table_time_errors}
\end{table}

The spatial discretization errors $\|u_h^N - u_{h,\rm ref}^N\|_{L^2(\Omega)}$ are presented in Table \ref{table_space_errors}, where the reference solution $u_{h,\rm ref}^N$ is chosen to be the numerical solution with mesh size $h_{\rm ref}=1/128$. We have chosen several time stepsizes to investigate the influence of temporal discretization on the spatial discretization errors. From Table \ref{table_space_errors} we see that the influence of temporal discretization can be neglected compared with the spatial discretization errors, which are $O(h^{1.5})$ in the $L^2$ norm. This is better than the result proved in Theorem \ref{THM:FEM-Euler} (similarly as the results shown in the previous example). 
\begin{table}[htb]\centering\small
\caption{Example \ref{example1}: Spatial discretization errors using variable stepsize with $\alpha=0.55$. }
\centering
\setlength{\tabcolsep}{4.0mm}{
\begin{tabular}{cccccc}
\toprule
$h$              &&   $\tau=1/80$&  $\tau=1/160$&  $\tau=1/320$&  $\tau=1/640$ \\[2pt]
\midrule
$1/8$            &&   1.0833E--03&   8.6397E--04&   7.4118E--04&   6.8490E--04 \\[2pt]
$1/16$           &&   4.1932E--04&   3.1724E--04&   2.5928E--04&   2.3009E--04 \\[2pt]
$1/32$           &&   1.4837E--04&   1.1103E--04&   8.9531E--05&   7.8055E--05 \\[2pt]
\midrule
Convergence rate && $O(h^{1.50})$& $O(h^{1.51})$& $O(h^{1.53})$& $O(h^{1.56})$ \\[2pt]
\bottomrule
\end{tabular}}
\label{table_space_errors}
\end{table}

\section{Conclusions}

We have presented an error estimate for a fully discrete semi-implicit Euler finite element method for the NS equations with $L^2$ initial data based on the natural regularity of the solution with singularity at $t=0$. The numerical solution is proved to be at least first-order convergent in both time and space without any CFL condition. The analysis makes use of the smoothing property of the NS equations under $L^2$ initial data and appropriate duality arguments to obtain a discrete $L^2(0,T_*;L^2)$ error bound for a sufficiently small constant $T_*$ (which depends on the initial data $u^0$, but independent of $\tau$ and $h$). 
This is proved by utilizing Lemma \ref{Lemma:FD-small}, which says that the discrete $L^2(0,T_*;L^2)$ norm of the numerical solution is not only bounded but also small for sufficiently small $T_*$. 
The discrete error bound in $L^2(0,T_*;L^2)$ is furthermore improved to $L^2(0,T;L^2)$ (for a general $T>0$) and a pointwise-in-time $L^2$ error bound away from $t=0$. The extension of the analysis to the Taylor--Hood finite elements (which do not satisfy property (P2)) is also possible.

Several questions still remain open for the NS equations with nonsmooth initial data. 

First, the numerical results show that $1.5$th-order convergence is achieved in the space discretization. This is slightly better than the result proved in this article. The proof of this sharper convergence rate still remain open. 

Second, the current numerical method and its error analysis requires variable stepsize and yields an error bound which holds only away from $t=0$. The development of efficient numerical methods that may have some uniform temporal convergence up to $t=0$ is still challenging. In view of the low-regularity integrators recently developed for dispersive equations \cite{Hofmanovaa-Schratz-2017,Ostermann-Schratz-FoCM-2018,Rousset-Schratz-2020} and semilinear parabolic equations \cite{Li-Ma-JSC} this is possible and worth to be considered (at least for semi-discretization in time). 

Third, the development of numerical methods with higher-order convergence (e.g., away from $t=0$) for the NS equations with initial data in critical spaces is still challenging and worth to be studied. 

Fourth, the error analysis of numerical methods for the three-dimensional NS equations with critical initial data in $H^\frac12$ or $L^3$ still remains open.

\bibliographystyle{abbrv}
\bibliography{NS}

\newpage

\appendix 
\section{Proof of Lemma \ref{lemma:regularity-u}}
\label{section:AppendixA}

We only need to prove \eqref{Eq:ut_energy_eq} by assuming that the initial value is in $H^2(\Omega)^2$, provided that all the constant $C_m$ below depends only on $m$ and $\|u^0\|_{L^2}$ (independent of higher regularity of $u^0$). 
Then, for a nonsmooth initial value $u^0\in \dot L^2$, we can choose a sequence of functions $u^0_{n}\in \dot H^1_0\cap H^2(\Omega)^2$, $n=1,2,\dots$, converging to $u^0$ in $\dot L^2$. The solution $u_n$ corresponding to the smooth initial value $u^0_{n}$ satisfies 
\begin{align}
\label{Eq:ut_energy_eq_un}
&\|\partial_t^m u_n(t)\|_{L^2}  + t^{\frac12}\|\partial_t^m u_n(t)\|_{H^1} 
+ t\|\partial_t^m u_n(t)\|_{H^2} \le C_m t^{-m} \quad\forall\, t>0,\,\,\, m\ge 0, 
\end{align} 
with a constant $C_m$ depending only on $m$ and $\|u^0_{n}\|_{L^2}$ (thus independent of $n$). By a standard compactness argument and passing to the limit $n\rightarrow\infty$, one obtains that $u_n(t)$ converges to $u(t)$ for $t>0$ and therefore \eqref{Eq:ut_energy_eq_un} implies \eqref{Eq:ut_energy_eq}. 

It remains to prove \eqref{Eq:ut_energy_eq} for $H^2$ initial value (thus the solution is qualitatively smooth in time and $H^2$ in space; see \cite[Remark 3.7]{Temam-1977}). 

From \eqref{Eq:u_energy_eq-L2} we immediately obtain  
\begin{equation}
\label{Eq:u_bound_by_init}
\|u\|_{L^\infty(0,\infty;L^2)} + \|u\|_{L^2(0,\infty;H^1)}\le C.
\end{equation}
%
To obtain higher-order estimates, we fix an arbitrary $s>0$ and let $\chi(t)$ be a nonnegative smooth cut-off function of time (independent of $x$) satisfying that 
\begin{align}\label{def-chi-t}
\chi(t)
=\left\{
\begin{aligned}
&0 &&\mbox{for}\,\,\,t\in(0,s/4),\\
&1 &&\mbox{for}\,\,\,t\in[s/2,\infty) ,
\end{aligned}
\right.
\quad\mbox{and}\quad
|\partial_t^k\chi(t)| \le Cs^{-k}\,\,\, \mbox{for}\,\,\, k=0,1,2,\dots
\end{align}
Testing \eqref{pde} by $\chi^2 \partial_tu$ yields 
\begin{align}\label{chi-dtu-L2-1}
&\|\chi\partial_tu\|_{L^2}^2
+\frac12 \chi^2  \frac{\d}{\d t}\|\nabla u\|_{L^2}^2 \notag\\
&= 
-(u\cdot\nabla u,\chi^2\partial_tu) \notag\\
&\le C\epsilon^{-1}\|u\cdot\nabla(\chi u)\|_{L^2}^2
+
\frac {\epsilon}{2}\|\chi\partial_tu\|_{L^2}^2 \\
&\le
C\epsilon^{-1}\|u\|_{L^4}^2\|\nabla(\chi u)\|_{L^4}^2
+
\frac {\epsilon}{2} \|\chi\partial_tu\|_{L^2}^2 
\notag\\
&\le
C\epsilon^{-1}\|u\|_{L^2}\|u\|_{H^1}
\|\nabla(\chi u)\|_{L^2}\|\Delta (\chi u)\|_{L^2}
+
\frac {\epsilon}{2}\|\chi\partial_tu\|_{L^2}^2
\quad\mbox{(here \eqref{L4-L2-H1}--\eqref{W14-H1-H2} are used)} \notag\\
&\le
C\epsilon^{-3}\|u\|_{L^2}^2\|u\|_{H^1}^2 \|\nabla(\chi u)\|_{L^2}^2
+
\frac {\epsilon}{2} \|\Delta (\chi u)\|_{L^2}^2 
+
\frac {\epsilon}{2}\|\chi\partial_tu\|_{L^2}^2  ,
\notag
\end{align}
where $\epsilon$ is an arbitrary positive constant arising from using Young's inequality. 
Since 
$$\chi^2  \frac{\d}{\d t}\|\nabla u\|_{L^2}^2=\frac{\d}{\d t} \|\nabla(\chi u)\|_{L^2}^2- 2\chi\partial_t\chi\|\nabla u\|_{L^2}^2$$ and $|\partial_t\chi|\le Cs^{-1}$, it follows from \eqref{chi-dtu-L2-1} that 
\begin{align}\label{Est1}
&\|\chi\partial_tu\|_{L^2}^2
+\frac12 \frac{\d}{\d t} \|\nabla(\chi u)\|_{L^2}^2 \notag \\
&\le
C\epsilon^{-3}\|u\|_{L^2}^2\|u\|_{H^1}^2 \|\nabla(\chi u)\|_{L^2}^2
+
\frac {\epsilon}{2}\|\Delta (\chi u)\|_{L^2}^2 
+
\frac {\epsilon}{2}\|\chi\partial_tu\|_{L^2}^2
+Cs^{-1}\|\nabla u\|_{L^2}^2.
\end{align}

To estimate the term $\epsilon\|\Delta (\chi u)\|_{L^2}^2$ on the right-hand side of \eqref{Est1}, we multiply \eqref{pde} by $\chi$ and consider the resulting equation
\begin{align*}
\left\{
\begin{aligned}
-\varDelta (\chi u)+\nabla (\chi p)
&=-\chi \partial_tu - u\cdot\nabla(\chi u) &&\mbox{in}\,\,\,\varOmega ,\\
\nabla\cdot(\chi u)&=0 &&\mbox{in}\,\,\,\varOmega ,\\
\chi u&=0 &&\mbox{on}\,\,\,\partial\varOmega .
\end{aligned}
\right.
\end{align*}
Through the standard  $H^2$ estimate of the linear Stokes equation (cf. \cite[Theorem 2]{Kellogg-Osborn-1976}) we obtain
\begin{align}\label{Est2}
\|\chi u\|_{H^2}^2
&\le C\|-\chi \partial_tu - u\cdot\nabla(\chi u)\|_{L^2}^2  \nonumber \\
&\le C\|\chi \partial_tu\|_{L^2}^2 + C\|u\|_{L^4}^2\|\nabla(\chi u)\|_{L^4}^2  \nonumber \\
&\le C\|\chi \partial_tu\|_{L^2}^2 + C\|u\|_{L^2}\|u\|_{H^1}\|\nabla(\chi u)\|_{L^2}\|\Delta (\chi u)\|_{L^2}  \nonumber \\
&\le
C\|\chi \partial_tu\|_{L^2}^2 +
C\|u\|_{L^2}^2\|u\|_{H^1}^2 \|\nabla(\chi u)\|_{L^2}^2
+
\frac {1}{2}\|\Delta (\chi u)\|_{L^2}^2 , 
\end{align}
where we have estimated the term $\|u\|_{L^4}^2\|\nabla(\chi u)\|_{L^4}^2$ similarly as in \eqref{chi-dtu-L2-1}. The last term of \eqref{Est2} can be absorbed by the left-hand side. Then adding $\epsilon\times$\eqref{Est2} to \eqref{Est1} yields 
\begin{align}\label{Est3}
&\|\chi\partial_tu\|_{L^2}^2
+ \frac12\frac{\d}{\d t} \|\nabla(\chi u)\|_{L^2}^2 
+ \epsilon\|\chi u\|_{H^2}^2 \nonumber \\ 
&\le
C\epsilon^{-3}\|u\|_{L^2}^2\|u\|_{H^1}^2 \|\nabla(\chi u)\|_{L^2}^2
+
\frac {\epsilon}{2}\|\Delta (\chi u)\|_{L^2}^2 
+C\epsilon \|\chi \partial_tu\|_{L^2}^2
+Cs^{-1}\|\nabla u\|_{L^2}^2 
. 
\end{align}
By choosing sufficiently small $\epsilon$, the two terms involving $\|\Delta (\chi u)\|_{L^2}^2$ and $\|\chi \partial_tu\|_{L^2}^2$ can be absorbed by the left-hand side. Then, by using estimate $\|u\|_{L^\infty(0,\infty;L^2)}\le C$ from \eqref{Eq:u_bound_by_init}, we obtain
\begin{align}\label{Est4}
&\frac12\|\chi\partial_tu\|_{L^2}^2
+ \frac12\frac{\d}{\d t} \|\nabla(\chi u)\|_{L^2}^2 
+ \frac{1}{2} \|\chi u\|_{H^2}^2 \le C\|u\|_{H^1}^2\|\nabla(\chi u)\|_{L^2}^2+Cs^{-1}\|\nabla u\|_{L^2}^2 .
\end{align}
Now, applying Gronwall's inequality, we have 
\begin{align}
&\|\nabla(\chi u)\|_{L^\infty(0,s;L^2)}^2 
+\|\chi\partial_tu\|_{L^2(0,s;L^2)}^2 
+ \|\chi u\|_{L^2(0,s;H^2)}^2 \\
&\le
\operatorname{exp}\left(C\|u\|_{L^2(0,s;H^1)}^2\right) 
Cs^{-1}\|\nabla u\|_{L^2(0,s;L^2)}^2 \nonumber \\ 
&\le Cs^{-1} , \nonumber 
\end{align}
where we have used the estimate $\|u\|_{L^2(0,\infty;H^1)}\le C$ from \eqref{Eq:u_bound_by_init}. 
Since $s>0$ is arbitrary and $\chi(t)=1$ for $t\ge s/2$, choosing $s=t$ in the inequality above yields that 
\begin{align}\label{Est5}
&\|u\|_{L^\infty(t/2,t;H^1)} 
+\|\partial_t u\|_{L^2(t/2,t;L^2)} 
+ \|u\|_{ L^2(t/2,t;H^2)} 
\le
Ct^{-\frac12}, \quad \forall t>0.  
\end{align}

We consider the mathematical induction on $m$, assuming that 
\begin{align}\label{mathind-1}
\|u^{(j)}(t)\|_{L^2} +  t^{\frac12} \big(\|u^{(j)}(t)\|_{H^1} 
+ \|\partial_t u^{(j)}(t)\|_{L^2(t/2,t;L^2)}  + \|u^{(j)}(t)\|_{L^2(t/2,t;H^2)} \big) \le C t^{-j}, \\
\forall t>0,\,\,\, j=0,\dots,m-1, \notag 
\end{align}
which holds for $m=1$ in view of \eqref{Eq:u_bound_by_init} and \eqref{Est5}. 

We denote $u^{(m)}=\partial_t^mu$ and differentiate \eqref{pde} $m$ times. This yields  
\begin{equation}
\label{diff-pde}
\left \{
\begin{aligned} 
\partial_tu^{(m)} + 
\sum_{j=0}^m \left(\begin{subarray}{c}
\displaystyle m\\ 
\displaystyle j
\end{subarray}\right) u^{(j)}\cdot\nabla u^{(m-j)} -\varDelta u^{(m)}
+\nabla p^{(m)} &= 0 && \mbox{in}\,\,\,  \varOmega\times (0,T] ,\\
\nabla \cdot u^{(m)} &=0 && \mbox{in}\,\,\,  \varOmega\times (0,T] ,\\
u^{(m)} &=0 && \mbox{on}\,\,\, \partial\varOmega\times (0,T] .
\end{aligned}
\right .
\end{equation}
Testing this equation by $u^{(m)}$ yields
\begin{align}\label{pre-dtm-u}
\frac12 \frac{\d}{\d t} \|u^{(m)}\|_{L^2}^2 + \|\nabla u^{(m)}\|_{L^2}^2
&\le C\sum_{j=0}^m \|u^{(j)}\cdot\nabla u^{(m-j)}\|_{H^{-1}}\|u^{(m)}\|_{H^1_0} \notag  \\
&\le C\sum_{j=0}^m \|u^{(j)}\|_{L^4}^2 \|u^{(m-j)}\|_{L^4}^2
+\frac14 \|\nabla u^{(m)}\|_{L^2}^2 \notag \\
&\le C\sum_{j=0}^m \|u^{(j)}\|_{L^2}\|\nabla u^{(j)}\|_{L^2} \|u^{(m-j)}\|_{L^2}\|\nabla u^{(m-j)}\|_{L^2}
+\frac14 \|\nabla u^{(m)}\|_{L^2}^2,
\end{align}
where we have used the following fact to get the second to last inequality: 
\begin{align*}
\|u^{(j)}\cdot\nabla u^{(m-j)}\|_{H^{-1}} 
=& \sup_{v \in H^1_0,\, \|v\|_{H^1} = 1} (u^{(j)}\cdot\nabla u^{(m-j)}, v) \\
=&  \sup_{v \in H^1_0,\, \|v\|_{H^1} = 1} -( u^{(m-j)},u^{(j)}\cdot\nabla v)\\
\le& \|u^{(j)}\|_{L^4}\|u^{(m-j)}\|_{L^4}.  
\end{align*}
%
%
Substituting \eqref{mathind-1} into \eqref{pre-dtm-u} for $1\le j\le m-1$, we obtain 
\begin{align*}
\frac12 \frac{\d}{\d t} \|u^{(m)}\|_{L^2}^2 + \|\nabla u^{(m)}\|_{L^2}^2
&\le C\|u\|_{L^2}\|\nabla u\|_{L^2} \|u^{(m)}\|_{L^2}\|\nabla u^{(m)}\|_{L^2}
+Ct^{-2m-1} 
+\frac14 \|\nabla u^{(m)}\|_{L^2}^2 \\
&\le C\|u\|_{L^2}^2\|\nabla u\|_{L^2}^2 \|u^{(m)}\|_{L^2}^2
+Ct^{-2m-1} 
+\frac12 \|\nabla u^{(m)}\|_{L^2}^2. 
\end{align*}
Then, multiplying the inequality above by $\chi^2$ and using $\|u(t)\|_{L^2}\le C$, we have  %
\begin{align*}
\frac12 \frac{\d}{\d t} \|\chi u^{(m)}\|_{L^2}^2 + \frac12 \|\nabla(\chi u^{(m)})\|_{L^2}^2 
&\le C\|\nabla u\|_{L^2}^2 \|\chi u^{(m)}\|_{L^2}^2
+C\chi(t) t^{-2m-1}
+C|\partial_t\chi| \chi \|\partial_t u^{(m-1)}\|_{L^2}^2 .
\end{align*}
By using Gronwall's inequality, estimate \eqref{Eq:u_bound_by_init} and property \eqref{def-chi-t}, we derive that 
\begin{align*}
 \|\chi u^{(m)}\|_{L^\infty(0,s;L^2)}^2& + \|\nabla(\chi u^{(m)})\|_{L^2(0,s;L^2)}^2\\
& \le \exp(C\|\nabla u\|_{L^2(0,s;L^2)}^2) \int_0^{s}(\chi(t)^2 t^{-2m-1} +C|\partial_t\chi(t)| \chi(t)\|\partial_t u^{(m-1)}\|_{L^2}^2)\d t \\
& \le \exp(C\|\nabla u\|_{L^2(0,s;L^2)}^2) \int_{\frac{s}{4}}^{s}(t^{-2m-1} +Cs^{-1}\|\partial_t u^{(m-1)}\|_{L^2}^2)\d t \\
& \le Cs^{-2m} .
\end{align*}
As a result, choosing $s=t$ in the inequality above, we have 
\begin{align}\label{reg-est}
 \|u^{(m)}\|_{L^\infty(t/2,t;L^2)} + \|\nabla u^{(m)}\|_{L^2(t/2,t;L^2)}
\le Ct^{-m}  ,\quad\forall\, t>0 .
\end{align}
%

From \eqref{diff-pde} we know that 
\begin{equation}
\left\{
\begin{aligned}
-\varDelta u^{(m)} + \nabla p^{(m)} &= 
- \partial_tu^{(m)} - 
\sum_{j=0}^m \left(\begin{subarray}{c}
\displaystyle m\\ 
\displaystyle j
\end{subarray}\right) u^{(j)}\cdot \nabla u^{(m-j)}, &&\mbox{in}\,\,\,\varOmega , \\
\nabla\cdot u^{(m)}&=0 &&\mbox{in}\,\,\,\varOmega ,\\[5pt]
u^{(m)}&=0 &&\mbox{on}\,\,\,\partial\varOmega .
\end{aligned}
\right.
\end{equation}
Through the standard $H^2$ estimate of linear Stokes equations (cf. \cite[Theorem 2]{Kellogg-Osborn-1976}) we obtain
\begin{align}\label{H2-udtm-1}
\|u^{(m)}\|_{H^2}^2
\le& 
C \|\partial_tu^{(m)} \|_{L^2}^2 
+
C\sum_{j=0}^m\|u^{(j)}\cdot \nabla u^{(m-j)}\|_{L^2}^2.
\end{align}
Testing equation \eqref{diff-pde} by $\partial_t u^{(m)}$ gives
\begin{align}\label{L2-dtudtm-1}
\|\partial_t u^{(m)}\|_{L^2}^2 + \frac12 \frac{\d}{\d t}\|\nabla u^{(m)}\|_{L^2}^2 
&\le C\sum_{j=0}^m \|u^{(j)}\cdot\nabla u^{(m-j)}\|_{L^2}\|\partial_t u^{(m)}\|_{L^2} \notag \\
&\le C\epsilon^{-1}\sum_{j=0}^m \|u^{(j)}\cdot\nabla u^{(m-j)}\|_{L^2}^2
+\epsilon\|\partial_t u^{(m)}\|_{L^2}^2 .
\end{align}
Summing up \eqref{L2-dtudtm-1} and $\lambda\times$\eqref{H2-udtm-1} yields 
\begin{align*}
&\|\partial_t u^{(m)}\|_{L^2}^2 + \frac12 \frac{\d}{\d t}\|\nabla u^{(m)}\|_{L^2}^2 +\lambda \|\Delta u^{(m)}\|_{L^2}^2 \\
&\le C\epsilon^{-1}\sum_{j=0}^m \|u^{(j)}\cdot\nabla u^{(m-j)}\|_{L^2}^2
+ (\epsilon+C\lambda) \|\partial_t u^{(m)}\|_{L^2}^2 \\
&\le C\epsilon^{-1}\sum_{j=0}^m \|u^{(j)}\|_{L^4}^2 \|\nabla u^{(m-j)}\|_{L^4}^2
+ (\epsilon+C\lambda)  \|\partial_t u^{(m)}\|_{L^2}^2 \\
&\le C\epsilon^{-1}\sum_{j=0}^m \|u^{(j)}\|_{L^2}\|\nabla u^{(j)}\|_{L^2} \|\nabla u^{(m-j)}\|_{L^2}\|\Delta u^{(m-j)}\|_{L^2}
+ (\epsilon+C\lambda)  \|\partial_t u^{(m)}\|_{L^2}^2 ,
\end{align*}
where we have used \eqref{L4-L2-H1}--\eqref{W14-H1-H2} in deriving the last inequality. By choosing sufficiently small $\epsilon$ and $\lambda$, the term $(\epsilon+C\lambda)  \|\partial_t u^{(m)}\|_{L^2}^2 $ can be absorbed by the left-hand side. 
Then, substituting \eqref{mathind-1} and \eqref{reg-est} into the inequality above, we obtain
\begin{align*}
&\frac12\|\partial_t u^{(m)}\|_{L^2}^2 + \frac12 \frac{\d}{\d t}\|\nabla u^{(m)}\|_{L^2}^2 +\lambda \|\Delta u^{(m)}\|_{L^2}^2  \\ 
&\le C\|u\|_{L^2} \|\nabla u\|_{L^2} \|\nabla u^{(m)}\|_{L^2}\|\Delta u^{(m)}\|_{L^2} + C\sum_{j=1}^{m-1} t^{-j-m-1} \|\Delta u^{(m-j)}\|_{L^2} \\
&\quad
+C\|u^{(m)}\|_{L^2}\|\nabla u^{(m)}\|_{L^2} \|\nabla u\|_{L^2}\|\Delta u\|_{L^2} \\
&\le C\lambda^{-1}\|u\|_{L^2}^2\|\nabla u\|_{L^2}^2 \|\nabla u^{(m)}\|_{L^2}^2
+\frac{\lambda}{2}\|\Delta u^{(m)}\|_{L^2}^2 
+ C\sum_{j=1}^{m-1} (t^{-2m-2} + t^{-2j} \|\Delta u^{(m-j)}\|_{L^2}^2) \\
&\quad
+C\|\nabla u\|_{L^2}^2\|\nabla u^{(m)}\|_{L^2}^2
+C\|u^{(m)}\|_{L^2}^2 \|\Delta u\|_{L^2} ^2 \\
&\le C\lambda^{-1}\|\nabla u\|_{L^2}^2 \|\nabla u^{(m)}\|_{L^2}^2
+\frac{\lambda}{2}\|\Delta u^{(m)}\|_{L^2}^2 
+ C\sum_{j=1}^{m-1} (t^{-2m-2} + t^{-2j} \|\Delta u^{(m-j)}\|_{L^2}^2) \\
&\quad
+C\|\nabla u\|_{L^2}^2\|\nabla u^{(m)}\|_{L^2}^2
+Ct^{-2m} \|\Delta u\|_{L^2} ^2 .
\end{align*}
%
%
After absorbing $\frac{\lambda}{2}\|\Delta u^{(m)}\|_{L^2}^2 $ by the left-hand side, multiplying the last inequality by $\chi^2$ yields 
\begin{align}
&\frac{\d}{\d t}\|\nabla (\chi u^{(m)})\|_{L^2}^2
+\|\partial_t (\chi u^{(m)})\|_{L^2}^2+\lambda \|\Delta (\chi u^{(m)})\|_{L^2}^2 \notag \\ 
&\le C\chi^2(t)\bigg[ \sum_{j=1}^{m-1}  (t^{-2m-2} + t^{-2j} \|\Delta u^{(m-j)}\|_{L^2}^2) + t^{-2m} \|\Delta u\|_{L^2}^2 \bigg] 
\notag\\
&\quad\, 
+ C\|\nabla u\|_{L^2}^2 \|\nabla (\chi u^{(m)})\|_{L^2}^2 
+ C|\partial_t\chi(t)|\chi(t)\|\nabla u^{(m)}\|_{L^2}^2 + |\partial_t\chi|^2 \| u^{(m)}\|_{L^2}^2
 . \notag 
\end{align}
Then we apply Gronwall's inequality in the time interval $[0,s]$ and using the estimate $|\partial_t\chi(t)|\le Cs^{-1}$. This yields that  
\begin{align*}
\|\nabla (\chi& u^{(m)})\|_{L^\infty(0,s;L^2)}^2
+\|\partial_t (\chi u^{(m)})\|_{L^2(0,s;L^2)}^2 + \|\Delta(\chi u^{(m)})\|_{L^2(0,s;L^2)}^2\\
\le & \exp(C\|\nabla u\|_{L^2(0,s;L^2)}^2) 
\int_{\frac{s}{4}}^s 
C\bigg[ \sum_{j=1}^{m-1}  (t^{-2m-2} + t^{-2j} \|\Delta u^{(m-j)}\|_{L^2}^2) + t^{-2m} \|\Delta u\|_{L^2}^2 \bigg]
\d t\\
&+ \exp(C\|\nabla u\|_{L^2(0,s;L^2)}^2) \int_{\frac{s}{4}}^s Cs^{-1}\|\nabla u^{(m)}\|_{L^2}^2 \d t \\
\le& Cs^{- 2m -1} ,
\end{align*}
where the last inequality uses \eqref{reg-est}. 
Since $s>0$ is arbitrary, choosing $s=t$ in the inequality above yields that 
\begin{align}\label{reg-est-nabla_u}
\|\nabla u^{(m)}\|_{L^\infty(t/2,t;L^2)}
+ \|\partial_t u^{(m)}\|_{L^2(t/2,t;L^2)}+ \|\Delta u^{(m)}\|_{L^2(t/2,t;L^2)}
&\le
Ct^{-m - \frac12}.
\end{align}
Combining \eqref{reg-est} and \eqref{reg-est-nabla_u}, we have
\begin{align}\label{mathind-f}
\|u^{(m)}(t)\|_{L^2} + t^{\frac12}\big( \|\partial_t u^{(m)}\|_{L^2(t/2,t;L^2)} + \|u^{(m)}(t)\|_{H^1} 
+ \|u^{(m)}(t)\|_{L^2(t/2,t;H^2)} \big) \le C t^{-m}, \, \forall\,  t >0 .
\end{align}
This completes the mathematical induction on \eqref{mathind-1}. Hence, \eqref{reg-est-nabla_u} holds for all $m$. 

By substituting estimates \eqref{Est5} and \eqref{mathind-f} into h\eqref{Est2} and considering $m=1$, we furthermore derive that 
\begin{align}\label{Est-chiu-H2}
\|u(t)\|_{H^2}
&\le
C\|\partial_tu(t)\|_{L^2} +
C\|u(t)\|_{L^2} \|u(t)\|_{H^1}^2 
\le Ct^{-1}, \quad \forall\,  t >0 .
\end{align}
From \eqref{H2-udtm-1} we also obtain that 
\begin{align*}
\|\Delta u^{(m)}\|_{L^2}
\le& 
C\|u^{(m+1)} \|_{L^2} 
+
C\sum_{j=0}^m\|u^{(j)}\cdot \nabla u^{(m-j)}\|_{L^2} \notag \\
\le& 
Ct^{-m-1} 
+
C\sum_{j=0}^m\|u^{(j)}\|_{L^4} \|\nabla u^{(m-j)}\|_{L^4} \notag \\
\le& 
Ct^{-m-1} 
+
C\sum_{j=0}^m\|u^{(j)}\|_{L^2}^{\frac12}\|\nabla u^{(j)}\|_{L^2}^{\frac12} \|\nabla u^{(m-j)}\|_{L^2}^{\frac12}\|\Delta u^{(m-j)}\|_{L^2}^{\frac12} \notag \\
\le& 
Ct^{-m-1} 
+
C\sum_{j=0}^m t^{-\frac{ j+m+1}{2}} \|\Delta u^{(m-j)}\|_{L^2}^{\frac12} \notag \\
\le& 
Ct^{-m-1} 
+
C\sum_{j=1}^m t^{-\frac{ j+m+1}{2}} \|\Delta u^{(m-j)}\|_{L^2}^{\frac12}
+ Ct^{-\frac{ m+1}{2}} \|\Delta u^{(m)}\|_{L^2}^{\frac12} .
\end{align*}
Assuming that $\|\Delta u^{(j)}\|_{L^2}\le C t^{-j-1}$ for $j=0,\dots,m-1$ (which holds for $m=1$ in view of \eqref{Est-chiu-H2}), the last inequality furthermore implies that 
\begin{align}\label{H2-u-dtm-f}
\|\Delta u^{(m)}\|_{L^2}
\le& 
C t^{-m-1}. 
\end{align}
By mathematical induction, \eqref{H2-u-dtm-f} holds for all $m\ge 0$. 

Combining \eqref{mathind-f} and \eqref{H2-u-dtm-f}, we obtain the desired result of Lemma \ref{lemma:regularity-u}. 
\hfill\endproof

\section{Proof of (\ref{converg-utauh-0})--(\ref{converg-utauh-2})}\label{section:AppendixB}

\noindent From \eqref{semi-Euler} we see that  
\begin{align*}
\| A_h u_h^n \|_{L^2}
&\le
\bigg\|\frac{ u_h^n - u_h^{n-1} }{\tau_n} \bigg\|_{L^2}
+\|u_h^{n-1}\|_{L^4} \|\nabla u_h^n\|_{L^4} \notag \\
&\le
\bigg\|\frac{ u_h^n - u_h^{n-1} }{\tau_n} \bigg\|_{L^2}
+\|u_h^{n-1}\|_{L^2}^{\frac12}  \|u_h^{n-1}\|_{H^1}^{\frac12} \|u_h^{n}\|_{H^1}^{\frac12} \|A_hu_h^n\|_{L^2}^{\frac12} \notag \\
&\hspace{82pt}\mbox{(here \eqref{L4-L2-H1} and Lemma \ref{Lemma:discrete_W14} are used)} \\ 
&\le
\bigg\|\frac{ u_h^n - u_h^{n-1} }{\tau_n} \bigg\|_{L^2}
+\frac12\|u_h^{n-1}\|_{L^2} (\|u_h^{n-1}\|_{H^1}^2+\|u_h^{n}\|_{H^1}^2)
+\frac12 \| A_h u_h^n \|_{L^2}. 
\end{align*}
As a result, we have 
\begin{align}\label{Ahuhn-L2-1}
\| A_h u_h^n \|_{L^2}
&\le
2\bigg\|\frac{u_h^n - u_h^{n-1}}{\tau_n} \bigg\|_{L^2} 
+\|u_h^{n-1}\|_{L^2} (\|u_h^{n-1}\|_{H^1}^2+\|u_h^{n}\|_{H^1}^2) . 
\end{align} 
Testing \eqref{semi-Euler} by $(u_h^n - u_h^{n-1})/ \tau_n$ yields 
\begin{align*}
&\bigg\|\frac{u_h^n - u_h^{n-1}}{\tau_n}\bigg\|_{L^2}^2
+\frac{\|\nabla u_h^n\|_{L^2}^2-\|\nabla u_h^{n-1}\|_{L^2}^2}{2\tau_n}
+\frac{\tau_n}{2}\bigg\|\frac{\nabla(u_h^n-u_h^{n-1})}{\tau_n}\bigg\|_{L^2}^2 \notag \\
&= -
\bigg(u_h^{n-1}\cdot\nabla u_h^n,   \frac{u_h^n - u_h^{n-1}}{\tau_n}\bigg) \notag \\
&\le \|u_h^{n-1}\|_{L^4} \|\nabla u_h^n\|_{L^4} 
\bigg\|\frac{u_h^n - u_h^{n-1}}{\tau_n}\bigg\|_{L^2} \notag \\
&\le \|u_h^{n-1}\|_{L^2}^{\frac12} \|u_h^{n-1}\|_{H^1}^{\frac12} \|u_h^{n}\|_{H^1}^{\frac12} \|A_hu_h^n\|_{L^2}^{\frac12} 
\bigg\|\frac{u_h^n - u_h^{n-1}}{\tau_n}\bigg\|_{L^2} \notag \\
&\hspace{65pt}\mbox{(again, \eqref{L4-L2-H1} and Lemma \ref{Lemma:discrete_W14} are used)} \\ 
&\le C\|u_h^{n-1}\|_{L^2}^{\frac12} \|u_h^{n-1}\|_{H^1}^{\frac12} \|u_h^{n}\|_{H^1}^{\frac12} \bigg\|\frac{u_h^n - u_h^{n-1}}{\tau_n}\bigg\|_{L^2}^{\frac32} 
+C\|u_h^{n-1}\|_{L^2} \|u_h^{n-1}\|_{H^1}^{\frac32} \|u_h^{n}\|_{H^1}^{\frac12} \bigg\|\frac{u_h^n - u_h^{n-1}}{\tau_n}\bigg\|_{L^2} \\
&\quad 
+C\|u_h^{n-1}\|_{L^2} \|u_h^{n-1}\|_{H^1}^{\frac12} \|u_h^{n}\|_{H^1}^{\frac32} \bigg\|\frac{u_h^n - u_h^{n-1}}{\tau_n}\bigg\|_{L^2} 
\qquad \mbox{(here \eqref{Ahuhn-L2-1} is used)} \notag \\
&\le
C\|u_h^{n-1}\|_{L^2}^2 (\| u_h^{n-1}\|_{H^1}^4+\| u_h^{n}\|_{H^1}^4) 
+ \frac12 \bigg\|\frac{u_h^n - u_h^{n-1}}{\tau_n}\bigg\|_{L^2}^2 . 
\end{align*}
The last term can be absorbed by the left-hand side. Then, multiplying the result by the smooth cut-off function $\chi(t_n)$ in \eqref{def-chi-t-1}, and using the estimate $\max\limits_{1\le n\le N}\|u_h^{n}\|_{L^2}\le C$, we obtain 
\begin{align*} 
&\chi(t_n)\bigg\|\frac{u_h^n - u_h^{n-1}}{\tau_n}\bigg\|_{L^2}^2 
+\frac{\chi(t_n)\|\nabla u_h^n\|_{L^2}^2-\chi(t_{n-1})\|\nabla u_h^{n-1}\|_{L^2}^2}{\tau_n} \notag \\ 
&\le
C \chi(t_n)( \|u_h^{n-1}\|_{H^1}^4 + \|u_h^{n}\|_{H^1}^4  )
+\frac{\chi(t_n)-\chi(t_{n-1})}{\tau_n} \|\nabla u_h^{n-1}\|_{L^2}^2  \notag \\
&\le
C\|\nabla u_h^{n-1}\|_{L^2}^2 \chi(t_n)\|\nabla u_h^{n-1}\|_{L^2}^2 
+C\|\nabla u_h^{n}\|_{L^2}^2 \chi(t_n)\|\nabla u_h^{n}\|_{L^2}^2  
+Ct_{m}^{-1} \|\nabla u_h^{n-1}\|_{L^2}^2 
\end{align*}
for $n\ge 2$ and $m\ge 4$ (so $\chi(t_1)=0$). 
By applying Gronwall's inequality, we obtain 
\begin{align*}
&\sum_{n=2}^N \tau_n
\chi(t_n) \bigg\|\frac{u_h^n - u_h^{n-1}}{\tau_n}\bigg\|_{L^2}^2
+\max_{2\le n\le N} \chi(t_n) \|\nabla u_h^n\|_{L^2}^2 \notag \\
&\le 
\exp\bigg(C\sum_{n=2}^N \tau_n  (\|\nabla u_h^{n-1}\|_{L^2}^2+\|\nabla u_h^{n}\|_{L^2}^2)\bigg)
Ct_{m}^{-1}  \sum_{n=2}^N \tau_n \|\nabla u_h^{n-1}\|_{L^2}^2 \notag \\ 
&\le C t_m^{-1} ,
\quad\mbox{where we have used the basic energy estimate \eqref{FD-basic-energy}}.
\end{align*}
Substituting this into \eqref{Ahuhn-L2-1}, we also obtain 
\begin{align*}
&\sum_{n=2}^N \tau_n \chi(t_n) \|A_hu_h^n\|_{L^2}^2 \le C t_m^{-1} . 
\end{align*}
To summarize, the two estimates above imply that 
\begin{align}
\label{FD-H1L2-LinftyH1}
&\sum_{n=[\frac{m}{2}]+1}^m \tau_n
\bigg( \|A_hu_h^n\|_{L^2}^2 + 
\bigg\|\frac{u_h^n - u_h^{n-1}}{\tau_n}\bigg\|_{L^2}^2 \bigg) 
+\max_{[\frac{m}{2}]< n\le m}  \|\nabla u_h^n\|_{L^2}^2 
\le C t_m^{-1}  
\quad\mbox{for}\,\,\, 4\le m\le N . 
\end{align} 

Let $u_{\tau,h}(t)$ be a piecewise linear function in time, defined by 
$$
u_{\tau,h}(t)=\frac{t_n-t}{\tau_n}u_h^{n-1}+\frac{t-t_{n-1}}{\tau_n}u_h^{n} \quad\mbox{for}\,\,\, t\in (t_{n-1},t_n] .
$$
Let $u_{\tau,h}^+(t)$ and $u_{\tau,h}^-(t)$ be piecewise constant functions in time, defined by 
$$
u_{\tau,h}^+(t)=u_h^{n}
\quad\mbox{and}\quad
u_{\tau,h}^-(t)=u_h^{n-1}  \quad\mbox{for}\,\,\, t\in (t_{n-1},t_n] .
$$
Then \eqref{FD-basic-energy} and \eqref{FD-H1L2-LinftyH1} imply that the following quantities remain bounded as $\tau,h\rightarrow 0$: 
\begin{align}
&\|u_{\tau,h}\|_{L^\infty(0,T;L^2)}
+\|u_{\tau,h}\|_{L^2(0,T;H^1)} 
+\|u_{\tau,h}^{\pm}\|_{L^\infty(0,T;L^2)}
+\|u_{\tau,h}^{\pm}\|_{L^2(0,T;H^1)} 
\le C ,\label{FD-LinftyL2-H1L2} \\
&\|\partial_tu_{\tau,h}\|_{L^2(T_1,T_2;L^2)}
+\|A_hu_{\tau,h}\|_{L^2(T_1,T_2;L^2)}
+\|u_{\tau,h}\|_{L^\infty(T_1,T_2;H^1)}
\le C ,\label{FD-H1L2-LinftyH1-L2H2} \\
&\|A_hu_{\tau,h}^{\pm}\|_{L^2(T_1,T_2;L^2)}
+\|u_{\tau,h}^{\pm}\|_{L^\infty(T_1,T_2;H^1)} 
\le C ,  \label{FD-Ahuh-L2L2-} 
\end{align}
for arbitrary fixed constants $T_1$ and $T_2$ such that $0<T_1<T_2\le T$. 
Since $$L^\infty(0,T;L^2(\Omega)^2)\cap L^2(0,T;H^1(\Omega)^2)\hookrightarrow L^4(0,T;L^4(\Omega)^2) ,$$ from \eqref{semi-Euler} we also derive that 
\begin{align}\label{dtu-tau-h-L2H-1}
\|\partial_tu_{\tau,h}\|_{L^2(0,T;\dot H^{-1})} 
&\le \|A_h u_{\tau,h}^+ \|_{L^2(0,T;\dot H^{-1})} 
+ C\|P_{X_h}(u_{\tau,h}^ -\cdot\nabla u_{\tau,h}^+)\|_{L^2(0,T;\dot H^{-1})} \notag \\ 
&\le C\|u_{\tau,h}^+ \|_{L^2(0,T;H^{1})} 
+ C\|u_{\tau,h}^ -\cdot\nabla u_{\tau,h}^+\|_{L^2(0,T;H^{-1})} \notag \\ 
&\le C\|u_{\tau,h}^+ \|_{L^2(0,T;H^{1})} 
+C\|u_{\tau,h}^-\|_{L^4(0,T;L^4)} \|u_{\tau,h}^+\|_{L^4(0,T;L^4)} \notag \\
&\le C. 
\end{align}
From \eqref{FD-LinftyL2-H1L2} and \eqref{dtu-tau-h-L2H-1} we see that 
$u_{\tau,h}$ is uniformly (with respect to $\tau$ and $h$) bounded in  $L^\infty(0,T; L^2(\Omega)^2)\cap L^2(0,T;\dot H^1_0)\cap H^1(0,T;\dot H^{-1})\hookrightarrow L^4(0,T;L^4(\Omega)^2)$, compactly embedded into $L^3(0,T;L^3(\Omega)^2)$ (see the Aubin--Lions--Simon theorem in \cite[Theorem 7]{simon-1986-compact}).
In the meantime, $u_{\tau,h}$ is uniformly bounded in  $H^1(T_1,T_2;L^2(\Omega)^2) \cap L^\infty(T_1,T_2;H^1(\Omega)^2)$, which is compactly embedded into $C([T_1,T_2];L^2(\Omega)^2)$ (cf. \cite[Theorem 5]{simon-1986-compact}) .
As a result, for any sequence $(\tau_j,h_j)\rightarrow (0,0)$ there exists a subsequence, also denoted by $(\tau_j,h_j)$ for the simplicity of notation, such that 
\begin{align}\label{subsequence-converg}
\begin{aligned}
&u_{\tau_j,h_j} \rightarrow u \quad\mbox{weakly$^*$ in}\,\,\, L^\infty(0,T;L^2(\Omega)^2), \\
&u_{\tau_j,h_j} \rightarrow u \quad\mbox{weakly in}\,\,\, L^2(0,T;\dot H^1_0), \\
&u_{\tau_j,h_j} \rightarrow u \quad\mbox{strong in}\,\,\, L^3(0,T;L^3(\Omega)^2) , \\
&u_{\tau_j,h_j} \rightarrow u \quad\mbox{weakly in}\,\,\, H^1(T_1,T_2;L^2(\Omega)^2) &&\mbox{for arbitrary $0<T_1<T_2\le T$}, \\
&u_{\tau_j,h_j} \rightarrow u \quad\mbox{weakly$^*$ in}\,\,\, L^\infty(T_1,T_2; H^1(\Omega)^2 ) &&\mbox{for arbitrary $0<T_1<T_2\le T$}, \\
&u_{\tau_j,h_j} \rightarrow u \quad\mbox{strongly in}\,\,\, C([T_1,T_2];L^2(\Omega)^2)
&&\mbox{for arbitrary $0<T_1<T_2\le T$},
\end{aligned}
\end{align}
for some function 
\begin{align*}
u
\in L^\infty(0,T;L^2(\Omega)^2)\cap L^2(0,T;\dot H^1_0) & \cap H^1(T_1,T_2;L^2(\Omega)^2)\cap L^\infty(T_1,T_2;H^1(\Omega)^2) \\
&\hookrightarrow C^{\frac12}([T_1,T_2];L^2(\Omega)^2) . 
\end{align*}

From \eqref{FD-LinftyL2-H1L2} we see that the set of functions $\{u_{\tau,h}(\cdot,t):t\in[0,T]\}$ is uniformly (with respect to $\tau$ and $h$) bounded in $\dot L^2$ and therefore precompact in $\dot H^{-1}$. From \eqref{dtu-tau-h-L2H-1} we also know that $u_{\tau,h}$ is uniformly bounded in $H^1(0,T;\dot H^{-1})\hookrightarrow C^{\frac12}([0,T];\dot H^{-1})$, which implies that the function $u_{\tau,h}:[0,T]\rightarrow \dot H^{-1}$ is equicontinuous with respect to $t\in[0,T]$. According to the Arzela--Ascoli theorem \cite[Chapter 7, Theorem 17]{Kelley-1975}, the functions $u_{\tau,h}$ are precompact in $C([0,T];\dot H^{-1})$ and therefore a subsequence $u_{\tau_j,h_j}$ satisfies that 
\begin{align}\label{converg-CH-1}
\mbox{$u_{\tau_j,h_j}$ converges to $u$ in $C([0,T];\dot H^{-1})$, 
with $u(0)=\lim\limits_{j\rightarrow \infty} P_{X_{h_j}}u^0 = u^0$ in $\dot H^{-1}$. }
\end{align}

It remains to prove that the limit function $u$ is the unique weak solution of the NS equations (thus the limit function is independent of the choice of the subsequence $\tau_j,h_j\rightarrow 0$). This would imply that $u_{\tau,h}$ converges to $u$ as $\tau,h\rightarrow 0$ in the sense of \eqref{subsequence-converg}--\eqref{converg-CH-1} without necessarily passing to a subsequence.

For $t\in(t_{n-1},t_n]\subset [T_1,T_2]$ we have 
$$
\|u_{\tau,h}^+-u_{\tau,h}\|_{L^2}
\le
\|u_h^n-u_h^{n-1}\|_{L^2} 
\le
\int_{t_{n-1}}^{t_n} \|\partial_tu_{\tau,h}\|_{L^2}\d t
\le 
C\tau_n^{\frac12}\|\partial_tu_{\tau,h}\|_{L^2(T_1,T_2;L^2)} . 
$$
As a result of \eqref{FD-H1L2-LinftyH1-L2H2}, we obtain 
$$
\|u_{\tau,h}^+-u_{\tau,h}\|_{L^\infty(T_1,T_2;L^2)}\rightarrow 0\quad\mbox{ 
as $\tau\rightarrow 0$} ,
$$
and similarly, 
$$
\|u_{\tau,h}^--u_{\tau,h}\|_{L^\infty(T_1,T_2;L^2)}\rightarrow 0\quad\mbox{ 
as $\tau\rightarrow 0$} .
$$
Hence, there exists a subsequence, also denoted by $(\tau_j,h_j)$ for the simplicity of notation, such that 
\begin{align*}
&u_{\tau_j,h_j}^{\pm} \rightarrow u \quad\mbox{strongly in}\,\,\, L^\infty(T_1,T_2; L^2(\Omega)^2) . 
\end{align*}
Note that 
\begin{align*}
\int_{0}^T \|u_{\tau,h}^{\pm}-u_{\tau,h}\|_{L^2}^2\d t
\le&
C \sum_{n=1}^N \tau_n \|u_h^n-u_h^{n-1}\|_{L^2}^2\\
\le&
C\tau_1 \|u_h^n-u_h^{n-1}\|_{L^2}^2  + C \sum_{n=2}^N \tau_n \|u_h^n-u_h^{n-1}\|_{\dot H^{-1}}\|u_h^n-u_h^{n-1}\|_{\dot H^1_0} \\
\le&
C\tau_1 + C \sum_{n=2}^N \tau_n^2 \|\partial_tu_{\tau,h}|_{[t_{n-1},t_n]}\|_{\dot H^{-1}}\|u_h^n-u_h^{n-1}\|_{H^1} \\
\le&
C\tau_1 + C\tau \|\partial_tu_{\tau,h}\|_{L^2(0,T;\dot H^{-1})} \bigg(\sum_{n=2}^N \tau_n (\|u_h^n\|_{H^{1}}^2+\|u_h^{n-1}\|_{H^{1}}^2) \bigg)^{\frac12}  \\
\le&C\tau \quad\mbox{(here \eqref{FD-basic-energy} and \eqref{dtu-tau-h-L2H-1} are used)}
\end{align*}
which immediately yields
$$
\|u_{\tau,h}^{\pm}-u_{\tau,h}\|_{L^2(0,T;L^2)}\rightarrow 0\quad\mbox{ 
as $\tau\rightarrow 0$} .
$$
Since $u_{\tau_j,h_j} \rightarrow u$ strongly in $L^2(0,T;L^2(\Omega)^2)$ as shown in \eqref{subsequence-converg}, it follows that 
\begin{align*}
&u_{\tau_j,h_j}^{\pm} \rightarrow u \quad\mbox{strongly in}\,\,\, L^2(0, T; L^2(\Omega)^2) . 
\end{align*}
From \eqref{FD-LinftyL2-H1L2} and \eqref{subsequence-converg} we know that, by passing to a subsequence if necessary, 
\begin{align*}
&u_{\tau_j,h_j}^{\pm} \rightarrow u \quad\mbox{weakly$^*$ in}\,\,\, L^\infty(0,T;L^2(\Omega)^2), \\
&u_{\tau_j,h_j}^{+} \rightarrow u \quad\mbox{weakly in}\,\,\, L^2(0,T;H^1(\Omega)^2) .
\end{align*}

Now, testing \eqref{semi-Euler} by $v_h\in C([0,T];X_h)$ and integrating the result in time, we have 
\begin{align*}
&\int_0^T (\partial_t  u_{\tau,h} ,v_h) \d t
+ \int_0^T(\nabla u_{\tau,h}^+,\nabla v_h) \d t 
+ \int_0^T (u_{\tau,h}^-\cdot\nabla u_{\tau,h}^+,v_h) \d t 
= 0  . 
\end{align*}
For any given $v\in C([0,T];\dot H^1_0\cap H^2(\Omega)^2)$, we let $v_h=\Pi_hv$ (see \eqref{Xh-X-approx}), which would converge to $v$ strongly in $C([0,T];\dot H^1)$ as $h\rightarrow 0$. Then the equation above implies that
\begin{align}\label{weak-form-u-tau-h}
&  \int_0^T (\partial_t u_{\tau,h} ,  v) \d t
+ \int_0^T(\nabla u_{\tau,h}^+,\nabla v) \d t 
+ \int_0^T (u_{\tau,h}^-\cdot\nabla u_{\tau,h}^+,v) \d t \notag \\ 
&= 
\int_0^T (\partial_t u_{\tau,h} , v-v_h) \d t
+ \int_0^T(\nabla u_{\tau,h}^+,\nabla(v-v_h)) \d t 
+ \int_0^T (u_{\tau,h}^-\cdot\nabla u_{\tau,h}^+,(v-v_h)) \d t \notag \\
&=:
J_1^h(v)+J_2^h(v)+J_3^h(v) ,
\end{align}
where 
\begin{align*}
|J_1^h(v)|
&= 
\bigg|\int_0^T (\partial_tu_{\tau,h} , v-v_h ) \d t\bigg| \notag \\
&\le
C\|\partial_t u_{\tau,h}\|_{L^2(0,T;\dot H^{-1})} \|v-v_h\|_{L^2(0,T;H^1)}
\rightarrow 0 ,\\
|J_2^h(v)|
&\le
C\|u_{\tau,h}^+\|_{L^2(0,T;H^1)} \|v-v_h\|_{L^2(0,T;H^1)}
\rightarrow 0 ,\\
|J_3^h(v)|
&= \bigg| \int_0^T (u_{\tau,h}^+,u_{\tau,h}^-\cdot\nabla  (v-v_h)) \d t \bigg| \\
&\le
C
\|u_{\tau,h}^+\|_{L^2(0,T;L^4)} \|u_{\tau,h}^-\|_{L^2(0,T;L^4)} 
\|\nabla(v-v_h)\|_{L^\infty(0,T;L^2)} \\
&\le
C
\|u_{\tau,h}^+\|_{L^2(0,T;H^1)} \|u_{\tau,h}^-\|_{L^2(0,T;H^1)} 
\|v-v_h\|_{L^\infty(0,T;H^1)} 
\rightarrow 0 . 
\end{align*}
Since $u_{\tau_j,h_j}, u_{\tau_j,h_j}^{\pm} \rightarrow u$ weakly in $L^2(0,T;\dot H^1_0)$ and $\partial_t u_{\tau_j,h_j} \rightarrow \partial_t u$ weakly in $L^2(0,T;\dot H^{-1})$, it follows that 
$$
\int_0^T (\partial_t  u_{\tau,h} ,v) \d t
+ \int_0^T(\nabla u_{\tau,h}^+,\nabla v) \d t 
\rightarrow
 \int_0^T (\partial_t  u , v) \d t
+ \int_0^T(\nabla u,\nabla v) \d t .
$$
Since $u_{\tau_j,h_j}^-\rightarrow u$ strongly in $L^2(0,T;L^2(\Omega)^2)$ and $\nabla u_{\tau_j,h_j}^{+} \rightarrow \nabla u$ weakly in $L^2(0,T;L^2(\Omega)^2)$, it follows that  $u_{\tau_j,h_j}^-\cdot \nabla u_{\tau_j,h_j}^+$ convergence weakly in $L^1(0,T;L^1(\Omega)^2)$ and therefore
$$
\int_0^T (u_{\tau,h}^-\cdot\nabla u_{\tau,h}^+,v) \d t 
\rightarrow
\int_0^T (u\cdot\nabla u,v) \d t, \,\,\, \forall \,\,v \in C([0, T]; \dot H^1_0\cap H^2(\Omega)^2) 
\hookrightarrow C([0, T];L^\infty(\Omega)^2) . 
$$
By using these results and passing to the limit $(\tau,h)=(\tau_j,h_j)\rightarrow (0,0)$ in \eqref{weak-form-u-tau-h}, we obtain that the limit function $u$ satisfies the following weak form: 
\begin{align}\label{weak-form-NS}
&\int_0^T (\partial_t  u , v) \d t
+ \int_0^T(\nabla u ,\nabla  v) \d t 
+ \int_0^T (u\cdot\nabla u,v) \d t =0
\quad\forall \,v \in C([0, T]; \dot H^1_0\cap H^2(\Omega)^2) . 
\end{align}
Note that $\partial_tu\in L^2(0,T;\dot H^{-1})$, $\nabla u\in L^2(0,T;L^2(\Omega)^2)$ and 
$$u\cdot\nabla u\in L^{\frac43}(0,T;L^{\frac43}(\Omega)^2)\subset L^{\frac43}(0,T; H^{-1}(\Omega)^2 ) \subset L^{\frac43}(0,T;\dot H^{-1}) .$$ 
Since $(P_{X}(u\cdot\nabla u),v) = (u\cdot\nabla u,v)$ for $v\in \dot H^1$, it follows that $P_{X}(u\cdot\nabla u)\in L^{\frac43}(0,T;\dot H^{-1}) $. Therefore, $\partial_t  u - Au + P_{X}(u\cdot\nabla u)\in L^{\frac43}(0,T; \dot H^{-1} )$ and \eqref{weak-form-NS} implies that 
$$
\int_0^T \big( \partial_t  u - Au + P_{X}(u\cdot\nabla u) ,v \big) \d t =0 
\quad\forall \, v \in C([0, T]; \dot H^1_0) . 
$$
This implies that 
\begin{align}\label{weak-form-NS-2}
\partial_t  u - Au + P_{X}(u\cdot\nabla u)=0 \,\,\,\,\, \mbox{in}\,\,\, \dot H^{-1} \,\,\,\mbox{a.e.}\,\,\, t\in(0,T]. 
\end{align}

From \eqref{FD-LinftyL2-H1L2} and \eqref{dtu-tau-h-L2H-1} we conclude that, as the limit of $u_{\tau_j,h_j}$ when $j\rightarrow\infty$, the limit function $u$ must satisfy 
\begin{align}\label{weak-form-NS-reg}
u\in L^\infty(0,T;\dot L^2)\cap L^2(0,T;\dot H^1_0)\cap H^1(0,T;\dot H^{-1})\hookrightarrow C([0,T];\dot L^2) . 
\end{align}
According to \cite[Problem 3.2 and Theorem 3.2 of Chapter 3]{Temam-1977}, 
the equation \eqref{weak-form-NS-2} and the regularity result \eqref{weak-form-NS-reg} imply that $u$ must be the unique weak solution of the 2D NS equation. 

Since every sequence $u_{\tau_j,h_j}$ contains a subsequence that converges to the unique weak solution $u$ in the sense of \eqref{subsequence-converg}--\eqref{converg-CH-1}, it follows that $u_{\tau,h}$ converges to $u$ as $\tau,h\rightarrow 0$ (without passing to a subsequence). 
Then \eqref{subsequence-converg}--\eqref{converg-CH-1} imply the desired results in (\ref{converg-utauh-0})--(\ref{converg-utauh-2}). 
\hfill\endproof
\bigskip\bigskip

\end{document}